\documentclass[a4paper,leqno,12pt,twoside]{amsart}
\usepackage{amsmath,amsfonts,amssymb,amsthm,amscd}
\usepackage[utf8]{inputenc}
\usepackage[T1]{fontenc}
\usepackage[english]{babel}
\usepackage[colorlinks=true,citecolor=blue, urlcolor=blue, linkcolor=blue,pagebackref]{hyperref}
\usepackage[top=1.2in,bottom=1.2in,left=1.in,right=1.in]{geometry} 
\usepackage{graphicx}
\usepackage{paralist}
\usepackage{tabto}
\usepackage{standalone}
\usepackage{tikz}
\usetikzlibrary{matrix}
\usetikzlibrary{arrows}
\usepackage{xfrac}
\usepackage{amsthm, amsmath, amssymb,latexsym}
\usepackage{tikz-cd}
\usepackage[all]{xy} 

\usepackage{enumerate}
\usepackage{xcolor}
\usepackage{aliascnt}
\usepackage{cleveref}

\newtheorem{theorem}{Theorem}[section]
\newtheorem{lemma}[theorem]{Lemma}
\newtheorem{corollary}[theorem]{Corollary}

\newtheorem{remark}[theorem]{Remark}
\newtheorem{definition}[theorem]{Definition}
\newtheorem{example}[theorem]{Example}

\def\ra{\rightarrow}
\title[Higher Gaussian maps]{Higher Gaussian maps on the hyperelliptic locus and second fundamental form}

\author{D. Faro}

\address{Dario Faro  \\ Universit\`a  degli Studi di Milano Statale  \\ Dipartimento di Matematica \\ Via Saldini 50  \\ 20123 Milano, Italy  }
 \email{dario.faro@unimi.it}

\author{P. Frediani}
\address{Paola Frediani  \\ Universit\`a degli Studi di Pavia  \\ Dipartimento di Matematica \\ Via Ferrata 1  \\ 27100 Pavia, Italy  }
 \email{paola.frediani@unipv.it}

 \author{A. Lacopo}
 \address{Antonio Lacopo \\ Universit\`a degli Studi di Pavia  \\ Dipartimento di Matematica \\ Via Ferrata 1  \\ 27100 Pavia, Italy  }
 \email{antonio.lacopo01@universitadipavia.it}

\begin{document}

\begin{abstract}
In this paper we study higher even Gaussian maps of the canonical bundle on hyperelliptic curves and we determine their rank, giving explicit descriptions of their kernels. Then we use this descriptions to investigate the hyperelliptic Torelli map $j_h$ and its second fundamental form.  We study isotropic subspaces of the tangent space $T_{{\mathcal H}_g, [C]}$ to the moduli space ${\mathcal H}_g$ of hyperelliptic curves of genus $g$ at a point $[C]$, with respect to the second fundamental form $\rho_{HE}$ of $j_h$. In particular, for any Weierstrass point $p \in C$, we construct a subspace $V_p$ of dimension $\lfloor\frac{g}{2} \rfloor$ of $T_{{\mathcal H}_g, [C]}$ generated by higher Schiffer variations at $p$, such that the only isotropic tangent direction $\zeta \in V_p$ for the image of $\rho_{HE}$ is the standard Schiffer variation $\xi_p$ at the Weierstrass point $p \in C$. 
\end{abstract}

\thanks
{D. Faro, P. Frediani and A. Lacopo are members of GNSAGA (INdAM) and are partially supported by PRIN project {\em Moduli spaces and special varieties} (2022).}

\maketitle

\section{Introduction}
The purpose of this paper is twofold. On the one hand we study higher even Gaussian maps of the canonical bundle on hyperelliptic curves and we determine their rank, giving an explicit description of their kernels. On the other hand, we use this description to  investigate the local geometry of the image via the Torelli map of the moduli space ${\mathcal H}_g$ of hyperelliptic curves of genus $g$ in the moduli space ${\mathcal A}_g$ of principally polarised abelian varieties of dimension $g$. 

Let $C$ be a hyperelliptic curve of genus $g \geq 3$. Denote by $I_2$ the kernel of the multiplication map 
$$\mu_0: S^2H^0(C,K_C)\rightarrow H^0(C,K_C^{\otimes 2})^+,$$
where $H^0(C,K_C^{\otimes 2})^+$ denotes the invariant subspace of $H^0(C,K_C^{\otimes 2})$ under the action of the hyperelliptic involution. The canonical map is the composition of the $2:1$ map $\pi: C \rightarrow {\mathbb P}^1$ with the $(g-1)-{th}$ Veronese embedding $\nu_{g-1}: {\mathbb P}^1 \rightarrow {\mathbb P}^{g-1}$, hence the space $I_2$ is identified with the vector space of quadrics containing the rational normal curve $\nu_{g-1}({\mathbb P}^1) $ in ${\mathbb P}^{g-1}$. 
The second Gaussian map $\mu_2$ is a linear map 
$$\mu_2: I_2 = Ker(\mu_0) \rightarrow H^0(C,K_C^{\otimes 4}),$$
and more generally, for any $k \geq 1$, the $2k$-th Gaussian map is a linear map
$$\mu_{2k}: Ker( \mu_{2k-2} ) \rightarrow H^0(C,K_C^{\otimes 2k+2}).$$
These maps were introduced by Wahl in \cite{wahl1}. One can also define odd Gaussian maps, each one being defined on the kernel of the previous one. The most important and studied one is the first Wahl map: 
$$\mu_1: \wedge^2 H^0(C, K_C) \rightarrow H^0(C, K_C^{\otimes 3}).$$

It was proven by Wahl that if $C$ sits on a $K3$ surface, then the first Wahl map $\mu_1$ is not surjective (see \cite{wahl1}, see also \cite{bm}). On the other hand, Ciliberto, Harris and Miranda proved that for the general curve of genus $g \geq 10$, $g \neq 11$, $\mu_1$ is surjective (see \cite{chm}, \cite{voi} for another proof). 

If $C$ is a hyperelliptic curve, it was proven in \cite{wahl}, \cite{cm} that the rank of $\mu_1$ is $2g-3$, while in \cite{colombofrediani} it was shown that the rank of $\mu_2$ is $2g-5$. 

Our first main result is the following 
\begin{theorem}
\label{TheoremA}
(See Theorem \ref{rteoremainrese})
Let $C$ be a hyperelliptic curve of genus $g \geq 3$ . Then for every $  0\leq k \leq  \frac{g-1}{2}$ 
\begin{equation}
Rank(\mu_{2k})=2g-(4k+1),
\end{equation}
\begin{equation}
    dim(Ker(\mu_{2k})) =\frac{(g-1)(g-2)}{2} - k(2g -2k -3).
\end{equation}
Then,  for every $k>  \lfloor\frac{g-1}{2}\rfloor$ the domain of $\mu_{2k}$ is $0$, hence  $Rank (\mu_{2k} ) =0$. 
\end{theorem}

From this we immediately see (Remark \ref{chain}) that if $g$ is odd we have the following chain of inclusions: $$
0=Ker(\mu_{g-1})\subsetneq Ker(\mu_{g-3})\subsetneq ...\subsetneq Ker(\mu_2)\subsetneq I_2;
$$
while if $g$ is even we have 
$$
0=Ker(\mu_{g-2})\subsetneq Ker(\mu_{g-4})\subsetneq ...\subsetneq Ker(\mu_2)\subsetneq I_2.
$$

Denote by ${\mathcal H}_g$ the hyperelliptic locus in ${\mathcal M}_g$ and by $$j_{h}: {\mathcal H}_g \rightarrow {\mathcal A}_g$$ the restriction of the Torelli map to ${\mathcal H}_g$. 

The map $j_h$ is an orbifold immersion (see \cite{os}). The main tool for studying the local geometry of the hyperelliptic Torelli locus $j_{h}({\mathcal H}_g) \subset {\mathcal A}_g$ is the second fundamental form of the map $j_h$. 
Recall that the moduli space ${\mathcal A}_g$ is the quotient of the Siegel space $Sp(2g, {\mathbb R})/U(g)$ under the action of $Sp(2g, {\mathbb Z})$. 
We endow ${\mathcal A}_g$ with the Siegel metric, that is the orbifold metric induced by the symmetric metric on the Siegel space. 
The second fundamental form of $j_h$ is a map 
$$\rho_{HE}: N^*_{\mathcal{H}_g|\mathcal{A}_g} \rightarrow Sym^2 \Omega^1_{\mathcal{H}_g},$$
 where $N^*_{\mathcal{H}_g|\mathcal{A}_g}$ denotes the conormal bundle of $j_h(\mathcal{H}_g)$ in $\mathcal{A}_g$. 
 At a point $[C] \in {\mathcal H}_g$, $\rho_{HE}$ is therefore a linear map 
 $$\rho_{HE} : I_2 \rightarrow Sym^2 (H^0(C, K_C^{\otimes 2})^+).$$

%Assume $|L|$ is  the $g^1_2$, fix a basis $\{ s, t \}$ of $H^0(C,L)$, and a basis $\{ \omega_1, ..., \omega_{g-1} \}$ of $H^0(C, K_C \otimes L^{\vee})$. Then a basis of $I_2$ is given by the quadrics $\{Q_{ij} := s\omega_i \odot t \omega_j - s \omega_j \odot t \omega_i\}$ (see Lemma \ref{fhgf}). 
One of the main problems in the study of the local geometry of the image of the hyperelliptic locus in ${\mathcal A}_g$ under the Torelli map is to investigate the existence of totally geodesic subvarieties of ${\mathcal A}_g$ contained in the hyperelliptic locus. 

This is related with the hyperelliptic Coleman-Oort conjecture, that says that for $g \geq 8$ there do not exist positive dimensional special subvarieties of ${\mathcal A_g}$ generically contained in the hyperelliptic Torelli locus, namely contained in the closure of $j_h({\mathcal H}_g)$ and intersecting $j_h({\mathcal H}_g)$. 
This conjecture was recently proven by Moonen \cite{moonenhyp}, generalising a previous result of Lu and Zuo \cite{lu-zuo}. 

In fact, special subvarieties are totally geodesic, hence the study of the second fundamental form is related with the above problem. 

Here we take a different viewpoint, namely, we are interested in the behaviour of the second fundamental form of the map $j_h$ in relation with higher even Gaussian maps.

%For $g \geq 8$, there are finitely many nonsingular complex hyperelliptic curves $C$ of genus $g$ for which $Jac(C)$ is a $CM $ abelian variety.
In \cite{cftrans} it was proven that the map $\rho_{HE}$  can be expressed in terms of the Hodge Gaussian map $\rho$ introduced in \cite{cpt} and that it is injective. 
An important property proven in \cite{cpt} is a formula computing, for any $Q \in I_2$, the Hodge Gaussian map $\rho(Q)$ on Schiffer variations at points $p $ on any curve $C$ in terms of the second Gaussian map $\mu_2(Q)$.  

This property is one main tool used in \cite[Theorem 6.2]{fp} to show that the maximal dimension of a germ of a totally geodesic subvariety of ${\mathcal A}_g$ contained in the hyperelliptic locus is $g+1$.

In the case of hyperelliptic curves, the result of \cite[section 5]{cftrans} gives an expression for $\rho_{HE}(Q)$ on the Schiffer variation $\xi_p \in H^1(C, T_C)^+$ at a Weierstrass point $p \in C$, in terms of the second Gaussian map. This allows to show that   $\rho_{HE}(Q)(\xi_p \odot \xi_p)=0$, for any Weierstrass poin $p$  (see Remark \ref{schasym}). 

In \cite{fredianihigher}, a relation between higher even
Gaussian maps $\mu_{2k}$ of the canonical bundle on a smooth projective curve of genus $g \geq 4$ and the Hodge Gaussian map $\rho$ is given, 
generalising the above result. 

Here we exploit similar techniques to compute  $\rho_{HE}(Q)$ for quadrics $Q$ contained in the kernel of higher Gaussian maps on  certain tangent directions given by odd higher Schiffer variations $\xi_p^{2k+1} \in H^1(C, T_C)^+$ at a Weierstrass point $p \in C$. We refer to Section \ref{Schiffer} for the definition of higher Schiffer variations. In a local coordinate centered at a  point $p$, the element $\xi_p^{n}  \in H^1(C, T_C)$ has a Dolbeault representative given by $\frac{ \bar{\partial}\rho_p}{z^n} \frac{\partial}{\partial z}$, where $\rho_p$ is a bump function in $p$ which is equal to one in a small neighborhood $U$ containing $p$. If $p$ is a Weierstrass point and $n$ is odd, one immediately sees that $\xi_p^n$ is invariant under the hyperelliptic involution, hence it is tangent to the hyperelliptic locus.

So, for any $k \leq 2g-3$, consider the $(k+1)$-dimensional subspace $$V_k:=\langle \xi_p^1,\xi_p^3,...,\xi_p^{2k+1}\rangle \subset H^1(T_C)^+.$$
We prove the following \begin{theorem} (See Theorems \ref{thm2}, \ref{thm3}). 
\label{thmB}
\begin{enumerate}
    \item Let $0\leq k\leq \lfloor\frac{g-3}{2}\rfloor$ and let $Q\in Ker(\mu_{2k})$. Then we have 
    \begin{equation*}
   \rho_{HE}(Q)(\xi_p^l\odot \xi_p^m)=0 \quad \forall l+m\leq 4k+3, \ l,m \ odd.
    \end{equation*}
    This implies that the subspace $V_k=\langle \xi_p^1,\xi_p^3,...,\xi_p^{2k+1}\rangle \subset H^1(T_C)^+$
    is isotropic for $\rho_{HE}(Q)$,  $\forall Q\in Ker(\mu_{2k}).$

  \item   There exists $Q \in Ker(\mu_{2k})$ such that $\rho(Q) ( \xi_p^{2k+1} \odot \xi_p^{2k +3}) \neq 0$. Hence the subspace $V_{k+1}=\langle \xi_p^1,...,\xi_p^{2k+1},\xi_p^{2k+3}\rangle  \subset H^1(T_C)^+$ is not isotropic for $\rho_{HE}(Ker(\mu_{2k}))$. 
  \end{enumerate}
\end{theorem}
Following \cite{cfp} we call 
a nonzero direction $\zeta \in H^1(T_C)^+$  asymptotic if 
$$\rho_{HE}(Q)(\zeta \odot \zeta) = 0, \ \forall Q \in I_2.$$
Clearly tangent directions to totally geodesic subvarieties are asymptotic directions.
Notice that saying that a nonzero element $\zeta \in H^1(T_C)^+$ is an asymptotic direction means that the point $[\zeta] \in {\mathbb P}H^1(T_C)^+$ is in the base locus of the linear space of quadrics $\rho_{HE}(I_2) \subset Sym^2 (H^1(T_C)^+)^{\vee}$. 
In \cite{cftrans} it is shown that $\rho_{HE}$ is injective, hence the image of $\rho_{HE}$ in $Sym^2 (H^1(T_C)^+)^{\vee}$ has dimension $\frac{(g-1)(g-2)}{2}$. So $\rho_{HE}(I_2)$ is a linear system of quadrics in ${\mathbb P}(H^1(C,T_C)^+ )\cong {\mathbb P}^{2g-2}$ of dimension  $\frac{(g-1)(g-2)}{2}$.

Thus if $g \leq 5$, there are asymptotic directions, in fact there are also examples of special subvarieties of ${\mathcal A}_g$ of positive dimension generically contained in the hyperelliptic locus (see e.g. \cite{fgp}). 

On the other hand, for $g$ sufficiently high one would expect that the intersection of a space of quadrics of dimension $\frac{(g-1)(g-2)}{2}$ in ${\mathbb P}^{2g-2}$ would be empty, hence that asymptotic directions would not exist. This is not true because Schiffer variations at Weierstrass points are asymptotic directions (see Remark \ref{schasym}). On the other hand we prove the following 

\begin{theorem} (See Theorem \ref{asymp}). 
\label{theoremC}

Let $C$ be a hyperelliptic curve of genus $g \geq 3$ and $p \in C$ be a Weierstrass point. Then the asymptotic directions in the space $V_{\lfloor \frac{g-2}{2} \rfloor }$  are exactly the Schiffer variations $\xi_p = \xi_p^1$.

\end{theorem}

%Let $C$ be hyperelliptic of genus $g \geq 3$ and $p \in C$ be a Weierstrass point. 
%
%METTERE $\lfloor\frac{g-2}{2} \rfloor$
%\begin{enumerate}
 %   \item  Assume $g$ is even. Then the asymptotic directions in the space $$V_{\frac{g-2}{2}}= \langle\xi_p^1,\xi_p^3,...,\xi_p^{g-1}\rangle \subset H^1(C, T_C)^+$$ are exactly the Schiffer variations $\xi_p = \xi_p^1$. 
 %   \item Assume $g$ is odd. Then the  asymptotic directions in the space $$V_{\frac{g-3}{2}}=\langle\xi_p^1,\xi_p^3,...,\xi_p^{g-2}\rangle \subset H^1(C, T_C)^+$$ are exactly the Schiffer variations $\xi_p = \xi_p^1$. 
%\end{enumerate}
    
%\end{theorem}

Finally, we show that Theorem \ref{asymp} also allows us to deduce Corollary \ref{bound}, which gives a bound for the dimension of a germ of a totally geodesic submanifold of ${\mathcal A}_g$, generically contained in the Hyperelliptic Torelli locus. Nevertheless this is weaker than the one already proven in \cite[Theorem 6.2]{fp}. 

The structure of the paper is as follows. 

In Section \ref{secidsyhbcvT} we recall the definition and basic properties of Gaussian maps. In Section \ref{section3} we determine the kernel of all even higher Gaussian maps on the hyperelliptic locus and we compute their rank, proving Theorem \ref{TheoremA}. In Section \ref{Schiffer} we define and give the basic properties of higher Schiffer variations. In Section \ref{second} we recall the definition and some results on the second fundamental form of the hyperelliptic Torelli map $j_h$. Finally in Section \ref{iso} we prove Theorems \ref{thmB}, \ref{theoremC} and their consequences (Remark \ref{chains} and Corollary  \ref{bound}).

\section{Gaussian  maps}
\label{secidsyhbcvT}
Let $C$ be a smooth curve of genus $g$, and let $L$ and $M$ be two line bundles on $C$.  In this section we recall the definition of the (higher) Gaussian maps on $C$ associated with the line bundles $L$ and $M$. 
\\
\

Set $S:=C\times C$ and let $\Delta
\subset S$ be the diagonal. Take a non-negative integer $k$, then  the
$k$-th Gaussian (or Wahl map) associated with $L$ and $M$ is
the map given by restriction to the diagonal
\begin{equation*}
  H^0(S,    L\boxtimes M (-k \Delta) ) \stackrel{\Phi_{k,L,M}}
  {\longrightarrow} H^0(S, L\boxtimes M (-k \Delta)_{|{\Delta}})\cong
  H^0(C, L\otimes M \otimes K_C^{\otimes k}).
\end{equation*}
We will only consider the case $L=M$, and we set   
$\Phi_{k,L}:=\Phi_{k,L,L}$.  Since we have the identification $H^0(S, L
\boxtimes M) \cong H^0(C,L) \otimes H^0(C,M)$, the map $\Phi_{0,L}$ is given by the multiplication map of global sections
\begin{equation*}
  H^0(C,L)\otimes
  H^0(C,L)\rightarrow H^0(C,L^{ \otimes 2}),
\end{equation*}
which vanishes identically on $\wedge^2 H^0(C,L)$.
Then we have 
$$\ker \Phi_{0,L} = H^0(S, L\boxtimes L(-\Delta)) \cong \wedge^2 H^0(C,L)\oplus I_2(L),$$
where $I_2(L)$ is the kernel of the multiplication map $\mu_{0,L}: S^2H^0(C,L)\rightarrow H^0(C,L^{\otimes 2})$. Since $\Phi_{1,L}$
vanishes on symmetric tensors, one  writes
\begin{equation}\nonumber
  \mu_{1,L}:= {\Phi_{1,L}}_{|\wedge^2H^0(L)}:\wedge^2H^0(L)\rightarrow H^0(
  K_C\otimes L^{ \otimes 2}).
\end{equation}
Take a local frame $l$ for $L$, local coordinate $z$ and two sections $s_1, s_2 \in H^0(C,L)$ with $s_i = f_i(z)l$. Then we have 
\begin{gather}
  \label{muformula}
  \mu_{1,L} (s_1\wedge s_2) = (f'_1 f_2 - f'_2f_1) dz \otimes l^{
    2} .
\end{gather}
Consequently, the zero divisor of $\mu_{1,L}(s_1\wedge s_2)$ is $2B +R$, where $B$ is the base locus of the pencil $\langle s_1, s_2 \rangle $ and $R$ is the ramification
divisor of the associated morphism.

Again $H^0(S, L\boxtimes L (-2\Delta))$ decomposes as the sum of
$I_2(L)$ and the kernel of $\mu_{1,L}$. Since $\Phi_{2,L}$ vanishes on skew-symmetric tensors, we write
\begin{equation*}
  \mu_{2,L}:= {\Phi_{2,L}}_{| I_2(L)}: I_2(L)\rightarrow H^0(C,L^{\otimes 2}\otimes K_C^{\otimes 2}).
\end{equation*}
We will denote by $\mu_2$  the second gaussian map of the canonical line
bundle $K_C$ on $C$:
\begin{equation*}
  \mu_2:= \mu_{2,K_C}:I_2(K_C)\rightarrow H^0(K_C^{\otimes 4}).
\end{equation*}
%\end{say}

In general, since even Gaussian maps vanish on antisymmetric tensors and odd Gaussian maps vanish on symmetric tensors,  for all $k \geq 1$, we set  
\begin{equation}
\mu_{2k,L}:={\Phi_{2k,L}}_{|_{Ker(\mu_{2k-2,L})}} : Ker(\mu_{2k-2,L}) \rightarrow H^0(L^{\otimes 2} \otimes K_C^{\otimes 2k}).
\end{equation}

\begin{equation}
\mu_{2k+1,L}:={\Phi_{2k+1,L}}_{|_{Ker(\mu_{2k-1,L})}} : Ker(\mu_{2k-1,L}) \rightarrow H^0(L^{\otimes 2} \otimes K_C^{\otimes 2k+1}).
\end{equation}

We will now give a local expression for these Gaussian maps. 

Choose a basis $\eta_1, ..., \eta_r$ of $H^0(L)$, whose local expressions is $\eta_i = f_i(z) l$, where $l$ is a local frame for $L$. Assume now that $Q = \sum_{ij} a_{ij} \eta_i \otimes \eta_j \in I_2(L)$, with $a_{ij} = a_{ji}$.  Then we have: 
$$ \sum_{ij} a_{ij} f_i(z) f_j(z) = 0,$$ 
since $Q \in I_2(L)$. Derivating we get  
$\sum_{ij} a_{ij} f'_i(z) f_j(z) = 0, $ and taking another derivative,  we obtain  
$$ \sum_{ij} a_{ij} f^{(2)}_i(z) f_j(z) = - \sum_{ij} a_{ij} f'_i(z) f'_j(z).$$

Then 
$$\mu_{2,L}(Q) =  \sum_{ij} a_{ij} f^{(2)}_i(z) f_j(z) dz^{\otimes 2} \otimes l^{\otimes 2}= - \sum_{ij} a_{ij} f'_i(z) f'_j(z) dz^{\otimes 2} \otimes l^{\otimes 2}.$$ 

In general $Q = \sum_{ij} a_{ij} \eta_i \otimes \eta_j \in Ker(\mu_{2k-2,L})$ if and only if 
\begin{equation}
\label{mu}
\sum_{i,j} a_{ij} f_i^{(h)}(z) f_j^{(n)} (z) \equiv 0, \ \ \forall h,n, \ \text{such that } \ h+n \leq 2k-1.
\end{equation}
Moreover, $\forall n =0, ..., 2k$,  we have: 
\begin{equation}
\label{rimu}
\mu_{2k,L}(Q) = (-1)^n \sum_{i,j}a_{ij} f_i^{(2k-n)}(z) f_j^{(n)} (z) dz^{\otimes 2k} \otimes l^{\otimes 2}.
\end{equation}

Analogously, $\alpha = \sum_{i<j} a_{ij} (\eta_i \wedge \eta_j )\in Ker(\mu_{2k-1, L})$ if and only if 
$$\sum_{i<j} a_{ij} (f_i^{(h)}(z) f_j^{(n)}-f_i^{(n)}(z) f_j^{(h) }) (z) \equiv 0, \ \ \forall h,n, \ \text{such that } \ h+n \leq 2k$$ 
and we have 
 
\begin{equation}
\mu_{2k+1,L}(\alpha )= (-1)^n \sum_{i<j}a_{ij} (f_i^{(2k+1-n)}(z) f_j^{(n)} (z) - f_j^{(2k+1-n)}(z) f_i^{(n)} (z) )dz^{\otimes 2k+1} \otimes l^{\otimes 2}.\end{equation}
for every $n=0,...,2k+1$.
%Moreover since $\alpha \in Ker(\mu^{2k-1}_L)$ if and only if 
%\begin{equation}
%\Phi^0_L(\alpha)=\mu^1_L(\alpha)=...=\mu^{2k-1}_L(\alpha)=\Phi^{2k}_L(\alpha)=0
%\end{equation} 
%and these %conditions are equivalent to the identies
%\begin{equation}
%\sum f_i^{(h)}f_j^{(k)} \equiv 0, \forall h,k \ \text{such that} \ h+k \leq 2k,
%\end{equation}
%one can equivalently express  \ref{mukdsd} as
%\begin{equation}
%\mu^{2k+1}_L(\alpha )=(-1)^{m}\sum (f_i^{(k+1+m)}g_i^{(k-m)}-f_i^{(k-m)}g_i^{(k+1+m)})dz^{\otimes 2} \otimes dz^{\otimes 2k+1}
%\end{equation}
%for every $m=0,...,k$.
%
\begin{remark}
\label{inclsudsds}
Observe that for every $k \geq 1 $ 
\begin{equation}
Ker(\Phi_{2k,L})=Ker(\mu_{2k,L}) \oplus Ker(\mu_{2k-1,L}),
\end{equation}
and from the definition of $\mu_{k,L}$ we have the inclusions:
\begin{align}
... &\subset Ker(\mu_{2k,L}) \subset Ker(\mu_{2(k-1),L}) \subset ... \subset Ker(\mu_{2,L}) \subset I_2=Ker(\mu_{0,L}); \\  \nonumber \\ 
...&\subset Ker(\mu_{2k+1,L}) \subset Ker(\mu_{2k-1,L}) \subset ... \subset Ker(\mu_{1,L}).
\end{align}
\end{remark}

\section{Rank of higher order Gaussian maps on hyperelliptic curves}
\label{section3}

Let $C$ be a hyperelliptic curve of genus $g \geq 3$, $L$ be the line bundle giving the $g^1_2$. Set $M = K_C \otimes L^{\vee}$ and $\pi:C \rightarrow {\mathbb P}^1$ the map induced by $|L|$. Denote by $\nu_n: {\mathbb P}^1 \hookrightarrow {\mathbb P}^{n}$ the $n^{th}$ Veronese embedding. Recall that the canonical map is given by the composition $\nu_{g-1} \circ \pi$, so $$K_C \cong \pi^*({\mathcal O}_{{\mathbb P}^1}(g-1)) \cong L^{\otimes (g-1)}$$ and $M \cong  L^{\otimes (g-2)}$. Then $H^0(C,M) \cong H^0({\mathbb P}^1, {\mathcal O}_{{\mathbb P}^1}(g-2))$  has dimension $g-1$.

Denote by $\Phi_k $, $\mu_k$, the Gaussian maps $\Phi_{k,K_C}$ and $\mu_{k,K_C}$. The main result of this section is the computation of the rank of higher Gaussian maps for any hyperelliptic curve of any genus. More precisely we are going to prove the following:
\begin{theorem}
\label{rteoremainrese}
Let $C$ be a hyperelliptic curve of genus $g \geq 3$ . Then for every $  0\leq k \leq  \frac{g-1}{2}$ 
\begin{equation}
Rank(\mu_{2k})=2g-(4k+1),
\end{equation}
\begin{equation}
    dim(Ker(\mu_{2k})) =\frac{(g-1)(g-2)}{2} - k(2g -2k -3).
\end{equation}
Then, $Rank (\mu_{2k} ) =0$ for every $k>  \lfloor\frac{g-1}{2}\rfloor$. 

\end{theorem}
Notice that the result was already known for %$\mu_1$ ??????and 
 $\mu_2$ (see \cite{colombofrediani}).  %(see \cite{wahl}, \cite{cm} for $\mu_1$, \cite{colombofrediani}  for $\mu_2$). 

\begin{remark}
\label{chain}
By Theorem \ref{rteoremainrese}, it is immediate to see that 
 if $g$ is odd we have the following chain of inclusions: $$
0=Ker(\mu_{g-1})\subsetneq Ker(\mu_{g-3})\subsetneq ...\subsetneq Ker(\mu_2)\subsetneq I_2(K_C);
$$
while if $g$ is even we have 
$$
0=Ker(\mu_{g-2})\subsetneq Ker(\mu_{g-4})\subsetneq ...\subsetneq Ker(\mu_2)\subsetneq I_2(K_C).
$$
Notice that if  $g$ is odd, $\dim(Ker(\mu_{g-3})) =1$, and if $g$ is even $\dim(Ker(\mu_{g-4})) =3$. 
\end{remark}
%The strategy of the proof is based on the fact that for a hyperelliptic curve we have an  explicit description of a basis of $H^0(\omega_C)$, and the strategy used by Colombo and Frediani in \cite{colombofrediani}  connecting the $\mu_2$ to the $\mu_L$ where $L \simeq \omega_C(-F)$ and $F$ is the $g^1_2$ on $C$. %in a way that it improves the possibility of making calculations.

Let us first recall the following well known result (see for example \cite{colombofrediani}).
\begin{lemma}
\label{fhgf}
Let $C$ be a hyperelliptic curve of genus $g \geq 3$. Let $|L|$ be the $g^1_2$. Set $M=K_C \otimes L^{\vee}$, let $\omega_1,...,\omega_{g-1}$ be a basis for $H^0(M)$ and let $\langle s,t \rangle$ be a basis for $H^0(L)$.  Then the map defined by 
\begin{align}
\Lambda^2H^0(M) &\xrightarrow{\psi} I_2 \nonumber
\\ \nonumber 
\omega_i \wedge \omega_j &\ra Q_{ij}:=s\omega_i \odot t\omega_j - s\omega_j \odot t\omega_i,\nonumber 
\end{align} 
%\begin{align}
%\Lambda^2H^0(L) &\ra I_2(K_X) \nonumber
%\\ \nonumber 
%t_i \wedge t_j &\ra Q_{ij}=tt_i \odot st_j - tt_j \odot st_i,\nonumber 
%\end{align} 
is an isomorphism. In particular, observe that $\{Q_{ij}\}$ gives a basis for $I_2$.
\end{lemma}
In \cite[Lemma 4.1]{colombofrediani} it is shown the following 
\begin{lemma}
\label{lemmadsds}
For any $Q_{ij}$ as in the previous lemma:
$$
\mu_2(Q_{ij})=\mu_{1,L}(s \wedge t)  \mu_{1,M}(\psi^{-1}(Q_{ij})),
$$ 
\end{lemma}
where with the expression \begin{equation}
\mu_{1,L}(s \wedge t)  \mu_{1,M}(\psi^{-1}(Q_{ij}))
\end{equation}
we mean the image of $\mu_{1,L}(s \wedge t) \otimes \mu_{1,M}(\psi^{-1}(Q_{ij}))$ under the multiplication map \begin{equation}
H^0(K_C \otimes L^{\otimes 2}) \otimes H^0(K_C \otimes M^{\otimes 2})  \ra H^0(K_C^{\otimes 4}).
\end{equation}
Observe that from Lemma \ref{lemmadsds} it follows that $Rank(\mu_2)=Rank(\mu_{1,M})$ since $\mu_{1,L}(s \wedge t)$  is a nonzero section in $H^0(K_C \otimes L^{\otimes 2})$. Indeed the zero locus of $\mu_{1,L}(s \wedge t)$  is given by the base locus of $|L|$ together with the ramification divisor of the induced morphism (recall \eqref{muformula}).
%(capire dove serve l'hp $ g \geq 3$ e se serve)
\\
\\
We start generalizing Lemma \ref{lemmadsds}  to any higher-order Gaussian maps. We use the same notations as in Lemma $\ref{fhgf}$
\begin{lemma}
\label{lemmasss}
%Let $k \geq 0$ be an integer, and let
%\begin{equation*}
%Q=\sum \limits_{1 \leq i < j \leq g-1}a_{ij}Q_{ij} \in Ker(\mu^{2k}).
%\end{equation*}
%Then 
%\begin{equation*}
%\sum \limits_{1 \leq i < j \leq g-1}a_{ij} (\omega_i \wedge \omega_j) \in Ker(\mu^{2k-1}_{L}),
%\end{equation*}
%where for $k=0$  we mean $\sum \limits_{1 \leq i < j \leq g-1}a_{ij} (\omega_i \wedge \omega_j) \in \Lambda^2H^0(L),
%$
%and
%$$
%\mu_{2k+2}(Q)=(k+1)\mu^1_{F}(s \wedge t) \mu_{2k+1,L} (\sum \limits_{1 \leq i < j \leq g-1}a_{ij} (\omega_i \wedge \omega_j)),
%$$
Let $k \geq 0$ be an integer and let 
\begin{equation*}
Q=\sum \limits_{1 \leq i < j \leq g-1}a_{ij}Q_{ij} \in Ker(\mu_{2k}).
\end{equation*}
Then
\begin{itemize}
\item[(i)]  for any $k \geq 1$ $$
\sum \limits_{1 \leq i < j \leq g-1}a_{ij} (\omega_i \wedge \omega_j)\in Ker(\mu_{2k-1,M}),
$$
\item[(ii)] for any $k \geq 0$
$$
\mu_{2k+2}(Q)=(k+1)\mu_{1,L}(s \wedge t) \mu_{2k+1,M} ( \sum \limits_{1 \leq i < j \leq g-1}a_{ij} (\omega_i \wedge \omega_j)) \in H^0(K_C^{\otimes{2k+4}}),$$
\end{itemize}
%\end{equation*}
\end{lemma}
%\blem
%Let $Q=\sum \limits_{1 \leq 1 < j \leq g-1}a_{ij}Q_{ij} \in Ker(\mu_{2k})$, $Q_{ij}$ as in \ref{fhgf}. Then
%\begin{itemize}
%\item[i)]  for any $k \geq 0$ $$
%\sum \limits_{1 \leq 1 < j \leq g-1}a_{ij} (\omega_i \wedge \omega_j))\in Ker(\mu_{2k+1,L});
%$$
%\item[ii)] for any $k \geq 0$
%$$
%\mu_{2k+2}(Q)=(k+1)\mu_{1,F}(s \wedge t) \mu_{2k+1,L} ( \sum \limits_{1 \leq 1 < j \leq g-1}a_{ij} (\omega_i \wedge \omega_j))$$;
%\item[iii)] for any $k \geq 0$ $Q \in Ker(\mu_{2k+2})$ if and only if $ \forall \  0 \leq m \leq k$, $\forall \  \text{max}\{3,2m+1\} \leq l \leq 2g-3$,
%$$
%\sum \limits_{\substack{1 \leq i < j \\ i+j=l}} a_{ij} ij(j-i)(i-1)(j-1)...(i-(\text{max}\{i-1,m-2\}))(j-(\text{max}\{j-1,m-2\})).
%$$
%\end{itemize}
%\elem
\begin{proof}
%Let us proceed by induction on $k$. Suppose $k=1$. 
%\begin{equation*}
%=0=\mu^2(Q)=\sum \limits_{1 \leq i < j \leq g-1}a_{ij}Q_{ij}\sum \limits_{1 \leq i < j \leq g-1}a_{ij}Q_{ij}=  \in Ker(\mu^{2k}).
%\end{equation*}
Let us proceed by induction. When $k=0$ $(ii)$ is Lemma \ref{lemmadsds}, and when $k=1$ $(i)$ it is an immediate consequence of the hypothesis $\mu_2(Q)=0$ together with Lemma \ref{lemmadsds}.
\\
\\
Now  take $n \geq 2$ and  suppose that $(ii)$ holds for every  $0 \leq k <n-1$ and $(i)$ holds for every  $1 \leq k <n$.  We are going  to prove that $(ii)$ holds for $k=n-1$, which automatically implies that  $(i)$ holds for $k=n$. Set $k=n-1$ and suppose that  
$$
Q=\sum \limits_{1 \leq 1 < j \leq g-1}a_{ij}Q_{ij} \in Ker(\mu_{2k}).
$$
On some open sets write
\begin{equation*}
\omega_i=f_idz\otimes T^{\vee},  \ t=fT \ \text{and} \ s=gT,
\end{equation*}
where $f_i$ is a holomorphic function and $T$ is a local generator on  $L$. From the definition of $Q_{ij}$ in Lemma \ref{fhgf}, it follows that  we can write locally  $Q$ as
\begin{equation}
\sum \limits_{1 \leq 1 < j \leq g-1}a_{ij}Q_{ij}=  \sum a_{ij} (ff_idz \odot gf_j dz - ff_j  dz \odot gf_i dz),
\end{equation}
where we can take $z$ to be a local coordinate on $C$. Then, by definition
\begin{align}
\label{eqmu2k}
&\mu_{2k+2}(\sum \limits_{1 \leq 1 < j \leq g-1}a_{ij}Q_{ij}) \nonumber
\\ %\nonumber
&=\sum \limits_{1 \leq 1 < j \leq g-1}a_{ij}((ff_i)^{(k+1)} (gf_j)^{(k+1)} - (ff_j)^{(k+1)} (gf_i)^{k+1}) dz^{\otimes (2k+4)}. 
\end{align}
Observe that the latter expression can be written as 
\begin{align}
&= \left[ \sum \limits_{1 \leq i < j \leq g-1}a_{ij} \sum\limits_{h=0}^{k+1} \binom{k+1}{h}f^{(k+1-h)}f_i^{(h)}  \sum\limits_{l=0}^{k+1} \binom{k+1}{l} g^{(k+1-l)}f_j^{(l)} \right] dz^{\otimes (2k+4)}\\ \nonumber
&-\left[\sum \limits_{1 \leq i < j \leq g-1}a_{ij}  \sum\limits_{h=0}^{k+1} \binom{k+1}{h} f^{(k+1-h)}f_j^{(h)}  \sum\limits_{l=0}^{k+1} \binom{k+1}{l}g^{(k+1-l)}f_i^{(l)}\right] dz^{\otimes (2k+4)}
\\
&=\left[\sum\limits_{h,l=0}^{k+1} \binom{k+1}{h} \binom{k+1}{l} f^{(k+1-h)} g^{(k+1-l)} \sum \limits_{1 \leq i < j \leq g-1}a_{ij} (f_i^{(h)} f_j^{(l)} -f_j^{(h)} f_i^{(l)})\right] dz^{\otimes (2k+4)} \label{dsrgfger}
\end{align}
By hypothesis $Q \in Ker (\mu_{2k})$ and hence from the inductive hypothesis
\begin{equation*}
\mu_{2k-1,M} (\sum \limits_{1 \leq i < j \leq g-1}a_{ij} (\omega_i \wedge \omega_j))=0.
\end{equation*}
Then using \eqref{mu} and \eqref{eqmu2k}  it follows that for every $h+l \leq 2k$
\begin{equation*}
\sum \limits_{1 \leq i < j \leq g-1}a_{ij}(f_i^{(h)} f_j^{(l)} - f_j^{(h)} f_i^{(l)}) \equiv 0.\end{equation*} Then  \eqref{dsrgfger} becomes
\begin{align*}
&\left[ \binom{k+1}{k} f g^{(1)} (\sum \limits_{1 \leq i < j \leq g-1}a_{ij} (f_i^{(k+1)} f_j^{(k)} -f_j^{(k+1)} f_i^{(k)})\right] dz^{\otimes (2k+4)} + \\ 
&\left[ \binom{k+1}{k}f^{(1)} g (\sum \limits_{1 \leq i < j \leq g-1}a_{ij} (f_i^{(k)} f_j^{(k+1)} -f_j^{(k)} f_i^{(k+1)}) \right] dz^{\otimes (2k+4)}.
\end{align*}
This is just
\begin{equation*}
 \left[\binom{k+1}{k} ( f g^{(1)} - f^{(1)}g)  (\sum \limits_{1 \leq i < j \leq g-1}a_{ij} (f_i^{(k+1)} f_j^{(k)} -f_j^{(k+1)} f_i^{(k)}) \right] dz^{\otimes (2k+4)},  
\end{equation*}
%\begin{equation}
%[\sum\limits_{h,l=k}^{k+1} c_{k+1,h} c_{k+1,l} f^{(k+1-h)} g^{(k+1-l)} \sum \limits_{1 \leq i < j \leq l-1}a_{ij} (f_i^{(h)} f_j^{(l)} -f_j^{(h)} f_i^{(l)})] dz^{\otimes (2k+4)},
%\end{equation}
which is equal to
$$
(k+1)\mu_{1,L}(t \wedge s)\mu_{2k+1,M}(\sum \limits_{1 \leq i < j \leq g-1}a_{ij} (\omega_i \wedge \omega_j)),
$$
and so $(ii)$ holds. 
\end{proof}
In the following lemma, we are going to describe the equations of the loci 
\begin{equation*}
   ...\subset Ker(\mu_{2k}) \subset Ker(\mu_{2(k-1)}) \subset ... \subset Ker(\mu_{2}) \subset I_2=Ker(\mu_{0}). 
\end{equation*}

Consider a hyperelliptic curve given by the affine equation:

\begin{equation}
\label{hypereq}
y^2 = \prod_{i=1}^{2g+2}(x-t_i).
\end{equation}

For simplicity from now on we will assume $t_1=0$, so $|M| =|K_C(-2p)|$, where $p$ is the Weierstrass point $(0,0)$ in  affine coordinates. 
Consider the basis 
\begin{equation}
\label{base}
\{\alpha_0=\frac{dx}{y},\alpha_1=\frac{xdx}{y},...,\alpha_{g-1}=\frac{x^{g-1}dx}{y}\}
\end{equation}

 of $H^0(K_C)$. Then a basis of $H^0(M)$ is given by 

\begin{equation}
\label{basedusygvh}
\{\alpha_1=\frac{xdx}{y},...,\alpha_{g-1}=\frac{x^{g-1}dx}{y}\}.
\end{equation}

 Set $S = C \times C$, and let $\Delta$ be the diagonal as usual. Then we have the following commutative diagram of exact sequences 
 
 \begin{equation}\label{bb}
\begin{tikzcd}[column sep=small, row sep=small, every node/.style={scale=0.9}]
  & 0 \arrow[r] &
    (M\boxtimes M)(-(k+1)\Delta) \arrow[d, hook] \arrow[r] &
    (M\boxtimes M)(-k\Delta) \arrow[d, hook] \arrow[r] &
    (M\boxtimes M)(-k\Delta)_{|\Delta} \arrow[d, hook] \arrow[r] &
    0 \\
  & 0 \arrow[r] &
    (K_C \boxtimes K_C)(-(k+1)\Delta) \arrow[r] &
    (K_C\boxtimes K_C)(-k\Delta) \arrow[r] &
    (K_C\boxtimes K_C)(-k\Delta)_{|\Delta} \arrow[r] &
    0 \\
\end{tikzcd}
\end{equation}

So, taking global sections we get 

\begin{equation}
 \begin{tikzcd}
 \label{bb1}
   & H^0(S,( M\boxtimes M)(-k\Delta))\ar[d, hook] \ar[r, "\Phi_{k,M}"] & H^0(S,( M\boxtimes M)(-k\Delta)_{|\Delta}\ar[d, hook] & \\
    & H^0(S,( K_C\boxtimes K_C)(-k\Delta)) \ar[r, "\Phi_{k}"] & H^0(S, ( K_C\boxtimes K_C)(-k\Delta)_{|\Delta})& \\
    & & & &  &\\
\end{tikzcd}
\end{equation}

Hence the corresponding diagram for the even order Gaussian maps:

\begin{equation}
 \begin{tikzcd}
 \label{bb2}
   & Ker(\mu_{2k-2, M})\ar[d, hook] \ar[r, "\mu_{2k,M}"] & H^0(C,M^{\otimes 2}\otimes K_C^{\otimes 2k})\ar[d, hook] & \\
    &Ker(\mu_{2k-2}) \ar[r, "\mu_{2k}"] & H^0(C,K_C^{\otimes 2k+2})& \\
    & & & &  &\\
\end{tikzcd}
\end{equation}

Hence  to compute the Gaussian maps $\mu_{2k, M}$, we compute the restriction of the Gaussian maps of the canonical bundle  $\mu_{2k}$ to the subspace $Ker(\mu_{2k-2, M})$ of $Ker(\mu_{2k-2})$.

The equations of $Ker(\mu_2)$ in $I_2$ in terms of this basis have already been described in the proof of \cite[Proposition 4.2]{colombofrediani}. Indeed we have the following
\begin{lemma}
\label{lemmanahMus}
Let 
\begin{equation*}
Q=\sum \limits_{1 \leq i < j \leq g-1}a_{ij}Q_{ij}  \in I_2,
\end{equation*}
$Q_{ij}$ as in \ref{fhgf}, with respect to the basis of $H^0(M)$ given  by $\alpha_i=x^i\frac{dx}{y}$,  for $i=1,...,g-1$. Then $Q \in Ker(\mu_2)$ if and only if for all $3 \leq l \leq 2g-3$
\begin{equation}
\label{dshbyeqyuo}
\sum \limits_{\substack{1 \leq i < j \\ i+j=l}} a_{ij} (j-i)=0.
\end{equation}
\end{lemma}
\
\\
Now we generalize the approach in \cite{colombofrediani}. The strategy is to use Lemma  \ref{lemmasss} together with the explicit description of the basis $\alpha_0,...,\alpha_{g-1}$ of the $H^0(C, K_C)$ for a hyperelliptic curve $C$. We have the following result.
\begin{lemma}
\label{lemmasadcx}
 Let 
 $$Q=\sum \limits_{1 \leq i < j \leq g-1}a_{ij}Q_{ij} \in I_2,
 $$
with  $Q_{ij}=s\alpha_i \odot t\alpha_j - s\alpha_j \odot t\alpha_i$ as in \ref{fhgf}, for $i=1,...,g-1.$ For any $k \geq 2$,  $Q \in Ker(\mu_{2k})$ if and only if 
\begin{itemize}
\item $\forall \  3\leq l \leq 2g-3$,
 $$
\sum \limits_{\substack{1 \leq i < j \\ i+j=l}} a_{ij} (j-i)=0,
$$
%which is condition \eqref{dshbyeqyuo};
\item  $\forall \  2 \leq m \leq k$, $\forall \  2m-1 \leq l \leq 2g-3$,
%$$
%\sum \limits_{\substack{1 \leq 1 < j \leq g-1 \\ i+j=l}} a_{ij} ij(j-i)(i-1)(j-1)...(i-(\text{max}\{i-1,m-2\}))(j-(\text{max}\{j-1,m-2\})).
%$$
\begin{equation}
\label{equationsmuk}
\sum \limits_{\substack{1 \leq i < j \\ i+j=l \\ i \geq m-1, \\ j \geq m-1 }} a_{ij} (j-i) ij(i-1)(j-1)...(i-(m-2))(j-(m-2))=0.
\end{equation}
\end{itemize}
\end{lemma}
\begin{proof}
We will argue locally.  Take the basis of $H^0(C,K_C)$ given in \eqref{base} and the one for $H^0(C,M)$ given in \eqref{basedusygvh}. We will proceed by induction on $k$. The thesis holds for $k=1$ by Lemma \ref{lemmanahMus}. Now we suppose that the thesis holds for $k-1$, $k \geq 2$ and take $Q \in Ker(\mu_{2k-2})$.  Since $Q$  also belongs to the previous kernels, 
%and
%4 \in $ By the inductive hypothesis we know that  $\forall \  2 \leq m \leq k-1$, $\forall \  2m-1 \leq l \leq 2g-3$,
%$$
%sum \limits_{\substack{1 \leq i < j \\ i+j=l \\ i \geq m-1, \\ j \geq m-1 }} a_{ij} (j-i) ij(i-1)(j-1)...(i-(m-2))(j-(m-2))=0.
%$$
proving the lemma  is   equivalent to prove that  $Q \in Ker(\mu_{2k})$ if and only if 
\begin{equation}
\sum \limits_{\substack{1 \leq i < j \\ i+j=l \\ i \geq k-1, \\ j \geq k-1 }} a_{ij} (j-i) ij(i-1)(j-1)...(i-(k-2))(j-(k-2))=0,
\end{equation}
for all $ 2k-1 \leq l \leq 2g-3$. Using lemma \ref{lemmasss} we have that  $\mu_{2k}(Q)=0$ if and only if $\mu_{2k-1,M} ( \sum \limits_{1 \leq i < j \leq g-1}a_{ij} (\alpha_i \wedge \alpha_j))=0$, that is if and only if
%\begin{equation}
%[\sum \limits_{1 \leq i < j \leq g-1}a_{ij} ((\frac{x^i}{y})^{(k-1)}(\frac{x^j}{y})^{(k)}-(\frac{x^i}{y})^{(k)}(\frac{x^j}{y})^{(k-1)})] dx^{\otimes (2k-1)} \otimes dz'^{\otimes 2}=0,
%\end{equation}
%where as usual $dz'$ is a local generator for $M$. Clearly, this is equivalent to show that 
\begin{equation}
\label{ewuyequat}
\sum \limits_{1 \leq i < j \leq g-1}a_{ij} ((x^i)^{(k-1)}(x^j)^{(k)}-(x^i)^{(k)}(x^j)^{(k-1)}) \equiv 0,
\end{equation}
where we are taking $x$ as local coordinate. Observe that \eqref{ewuyequat} is equal to 
\begin{align}
&\sum \limits_{\substack{1 \leq i < j \leq g-1\\ i \geq k-1 \\ j \geq k }}a_{ij}i...(i-(k-2))x^{(i-(k-1))}j...(j-(k-1))x^{j-k}
\\
&-\sum \limits_{\substack{1 \leq i < j \leq g-1\\ i \geq k \\ j \geq k-1 }}a_{ij}i...(i-(k-1))x^{(i-k)}j...(j-(k-2))x^{j-(k-1)},
\end{align}
which can be written after some simple algebraic manipulations as 
%\\
%&=\sum \limits_{\substack{1 \leq i < j \leq g-1\\ i \geq k-1 \\ j \geq k }}x^{(i+j-(2k-1))}a_{ij}i...(i-(k-2))j...(j-(k-1))
%\\
%&-\sum \limits_{\substack{1 \leq i < j \leq g-1\\ i \geq k \\ j \geq k-1 }}x^{(i+j-(2k-1))}a_{ij}i...(i-(k-1))j...(j-(k-2))
%\\
%&=\sum \limits_{\substack{1 \leq i < j \leq g-1\\ i \geq k \\ j \geq k }}x^{(i+j-(2k-1))}a_{ij}i...(i-(k-2))j...(j-(k-2))(j-(k-1)-(i-(k-1))
%\\
%&+\sum \limits_{\substack{1 \leq i < j \leq g-1\\ i=k-1 \\ j \geq k }}x^{(i+j-(2k-1))}a_{ij}i...(i-(k-2))j...(j-(k-1))
%\\
%&=\sum \limits_{\substack{1 \leq i < j \leq g-1\\ i \geq k \\ j \geq k }}x^{(i+j-(2k-1))}a_{ij}i...(i-(k-2))j...(j-(k-2))(j-i)
%\\
%&+\sum \limits_{\substack{1 \leq i < j \leq g-1\\ i=k-1 \\ j \geq k }}x^{(i+j-(2k-1))}a_{ij}i...(i-(k-2))j...(j-(k-1))
%\\
%&=
\begin{equation}
\sum \limits_{\substack{1 \leq i < j \leq g-1\\ i \geq k-1 \\ j \geq k-1 }}x^{(i+j-(2k-1))}a_{ij}i...(i-(k-2))j...(j-(k-2))(j-i).
\end{equation}
So we have shown that for $Q \in Ker (\mu_{2k-2}$),  $\mu_{2k}(Q)=0$ if and only if  for every $  2k-1 \leq l \leq 2g-3$
$$
\sum \limits_{\substack{1 \leq i < j \leq g-1\\ i+j=l \\ i \geq k-1 \\ j \geq k-1 }}a_{ij}(j-i)i...(i-(k-2))j...(j-(k-2))=0.
$$
This concludes the proof. 
\end{proof}
Now we come to the proof  Theorem \ref{rteoremainrese}. The strategy of the proof is clear: we have to determine the number of linearly independent equations in \eqref{equationsmuk}.
%\begin{proposition}
%\label{propositionsasa}
%Let $C$ be a hyperelliptic curve of genus $g \geq 3$. Let $ 1 \leq k \leq \frac{g-1}{2}$. Then dim(Ker($\mu_{2k}$)= dim Ker$(\mu_{2(k-1)}) - (2g-(4k+1)))$. 

%More precisely 
%$$
%\sum \limits_{\substack{1 \leq i < j \leq g-1\\ i+j=l \\i \geq k-1 \\ }} a_{ij} (j-i) ij(i-1)(j-1)()(i-(k-2))(j-(k-2))=0, \ \ \ \ 2k-1 \leq l \leq 2g-3
%$$
%impose $2g-(4k+1)$ linearly independent conditions. The ones given by
%$$
%\sum \limits_{\substack{1 \leq i < j \leq g-1\\ i+j=l \\i \geq k-1 \\ }} a_{ij} (j-i) ij(i-1)(j-1)...(i-(k-2))(j-(k-2))=0, \ \ \ \  2k+1 \leq l \leq 2g-(2k+1).
%$$
%If $k > \frac{g-1}{2}$, then $\mu_{2k}$ is identically zero. 
%\end{proposition}

\begin{proof}[Proof of Theorem \ref{rteoremainrese}]
    We will start introducing some notations. For every $i,j$ such that $1 \leq i < j \leq g-1$, for every $l=3,...,2g-3$, let $(w_{l})^{2}$ the vector whose coordinates are
$$
(w_{l})^{2}_{ij}=
\begin{cases}
(j-i) \ \ \text{if}\   i+j=l \\
0 \ \ \text{if} \ i+j \neq l
\end{cases}
$$
which are the coefficients of the $l$-th equation of the set of equations given in Lemma \ref{lemmasadcx}, describing $Ker(\mu_2)$, ordered by increasing values of $i$. For every $2 \leq r \leq k$ and  $l=2r-1,...,2g-3$, let $w_l^{(2r)}$ be the vector whose coordinates are
\begin{equation}
\label{fdfjbcseref}
(w_{l})^{(2r)}_{ij}=
\begin{cases}
(j-i)ij(i-1)(j-1)...(i-(r-2))(j-(r-2)) \ \ \text{if}\   i+j=l, i,j \geq r-1 \\ 
0 \ \ \text{if} \ i+j \neq l, \ \text{or} \  1 \leq i \leq r-2.
\end{cases}
\end{equation}
which  are the coefficients of the $l$-th equation of the set of equations which describe $Ker(\mu_{2r})$ inside $(Ker(\mu_{2(r-1)}))$, ordered by increasing values of $i$. 
\\
\\
For any $3 \leq l \leq 2g-3$ set
\begin{equation}
n_l=\#\{(i,j): 1\leq i <j \leq g-1,  \  i+j=l\}. 
\end{equation}
 and let $B'_{k,l}$ be the $k \times n_l$ matrix whose $r$th row, $ 1\leq r \leq k$, are the $n_l$ coordinates of  $w_{l}^{(2r)}$ corresponding to the indexes $(i,j)$ such that $i+j=l$. 
\\
\\
Observe that $B'_{k,l}$ is just the matrix with rows $w_l^{(2)},...,w_l^{(2k)}$ where we have removed the entries corresponding to $i+j \neq l$. Notice that these entries are all $0$ for all the  vectors $w_l^{(2r)}$.
\\
\\
Define $c_{k,l}=min\{n_l,k\}$ and let $B_{k,l}$ be the minor of  $B'_{k,l}$ of order $c_{k,l}$ given by the first $c_{k,l}$ rows and $c_{k,l}$ columns. We are going to prove that it is not zero for all $ 3\leq l \leq 2g-3$ by induction on $k$. 
\\
\\
Consider first the case $k=2$. Observe that if $l=3,4,2g-4,2g-3$, then $n_l=1$. In this case $B_{2,l}=j-i>0$ where $(i,j)$ is the only pair such that $i+j=l$. If $ 5 \leq l \leq 2g-5 $ then $n_l \geq 2$ and $c_{2,l}=min(2,n_l)=2$ and we have 
$$
B_{2,l}=
\begin{pmatrix}
    j-i & j-i-2 \\
    ij(j-i) & (j-i-2)(i+1)(j-1)
\end{pmatrix}
$$
where $i$ is the minimum such that $i+j=l$. Observe that $$
det(B_{2,l})=(j-i)(j-i-2)
det\begin{pmatrix}
    1 & 1 \\
    ij & (i+1)(j-1)
\end{pmatrix}
$$
Set
$$
A_{2,l}:=
\begin{pmatrix}
    1 & 1 \\
    ij & (i+1)(j-1)
\end{pmatrix}
$$
Subtracting the first column of $A_{2,l}$ from the second we obtain the matrix
$$
\begin{pmatrix}
    1 & 0 \\
    ij & j-i-1
\end{pmatrix}
$$
whose determinant is $j-i-1>0$. Now assume that for every  $ 2\leq  k_0 \leq k$ and for every $3 \leq l \leq  2g-3$, $B_{k,l}$ is not zero. We want to prove that the same holds for $k+1$. If $n_l \leq k$, then $n_l < k+1$,  $c_{k+1,l}=min\{n_l,k\}=n_l$ and  $B_{k+1,l}=B_{k,l}$. In this case, we have nothing to prove. Assume then $n_l > k$. Then $c_{k+1,l}=min\{n_l,k+1\}=k+1$. Then $B_{k+1,l}$ is the following $(k+1) \times (k+1)$ matrix
\begin{center}
{\tiny
$$
\begin{pmatrix}
j -i & j-i-2& ...& j-i-2k \\
(j -i)ji &(j-2-i)(i+1)(j-1)&...& (j-i-2k)(j-k)(i+k) \\
. &. &.& . \\
. &. & .&  . \\
. &. & .&  . \\
(j-i)ji...(j-(k+1-2))(i-(k+1-2)) & & ...& (j-i-2k)(j-k)(i+k)...(j-2k+1)(i+1)
\end{pmatrix}
$$}
\end{center}
where $i$ is the minimum such that $i+j=l$ and where we mean that every entry in the $r$-th row, $1\leq r \leq k+1 $ (corresponding to an index $(i,j)$) is zero whenever $i \leq r-1$.
%(recall the definition of $w^{(2r)}_l$, \ref{fdfjbcseref})

The determinant of the above matrix is equal to the product of  $(j-i)...(j-i-2k)$ by the determinant of the following matrix 
\begin{center}
{\tiny
$$
\begin{pmatrix}
1 & 1& ...& 1 \\
ji &(i+1)(j-1)&...& (j-k)(i+k) \\
. &. &.& . \\
. &. & .&  . \\
. &. & .&  . \\
(j-i)ji...(j-(k+1-2))(i-(k+1-2)) & & ...& (j-k)(i+k)...(j-2k+1)(i+1)
\end{pmatrix}
$$}
\end{center}

Now consider the matrix above and subtract each column to the previous one. Then one gets a matrix 
\begin{center}
{\tiny
$$
\begin{pmatrix}
1 & 0& ...& 0 \\
ji &(j-1)-i&...& (j-1)-i-2(k-1) \\
. &. &.& . \\
. &. & .&  . \\
. &. & .&  . \\
(j-i)ji...(j-(k+1-2))(i-(k+1-2)) & & ...& (j-1-i-2(k-1))(j-1-(k-1))(i+k-1)...(j-2(k-1)+1)(i+1)
\end{pmatrix}.
$$}
\end{center}
Observe that 
\begin{center}
{\tiny
$$
\begin{pmatrix}
(j-1)-i&...& (j-1)-i-2(k-1) \\
. &.& . \\
. & .&  . \\
. & .&  . \\
(j-1-i)(j-1)i...((j-1-(k-2)))(i-((k-2)))& ...& (j-1-i-2(k-1))(j-1-(k-1))(i+k-1)...(j-2(k-1)+1)(i+1)
\end{pmatrix}.
$$}
\end{center}
is $B_{k,l-1}$ (where as we have already said,  we mean that every entry corresponding to a index $(i,j)$ in the $r$-th row, $1\leq r \leq k $ is zero whenever $i \leq r-1$. The matrix has not zero determinant  by the inductive hypothesis. Hence we have shown, by induction,  that for any $ 3 \leq l \leq 2g-3$ and for any $k \geq 2$, $(w_{l})^{2}$,...,$(w_{l})^{2k}$ impose exactly $c_{k,l}$
 linearly independent conditions.
 \\
 \\
 Now let  $s_k$  be the number of  the indexes $l$ such that $c_{k,l} > c_{k-1,l}$. Observe that the condition  $c_{k,l} > c_{k-1,l}$ is equivalent to say that $l$-th equations of $Ker(\mu_{2k})$ are independent from the $l$-th equations of $Ker(\mu_{2}),...,Ker(\mu_{2(k-1)})$. Then we have
 \begin{equation}
    dim(Ker(\mu_{2k})= dim Ker(\mu_{2(k-1)}) - s_k. 
 \end{equation}
 Notice  that $c_{k,l} > c_{k-1,l}$  if and only if  $k-1 < n_l$. We claim that this happens if and only if $2k+1\leq l \leq 2g-(2k+1).$ 
 \\
 \\
In fact, if $2k+1\leq l \leq 2g-(2k+1)$, we have at least $k$ coordinates $(i,j)$ such that  $i+j=l$. Indeed write $l$ as  $l=2k+1+m$, $0 \leq m \leq 2g-2(2k+1)$ (where  we are using the assumption  $2g-2(2k+1) \geq 0$). 
 \\
 If $l$ is odd (equivalently $m$  is even), these are given by 
$$
(i,j)=(1+\frac{m}{2},2k+\frac{m}{2}),...,(k+\frac{m}{2},k+1+\frac{m}{2})
$$
These are easily seen to be  admissible indexes $(i,j)$, that is they satisfy $ 1 \leq i < j$, $i+j=l$ and $j \leq g-1$. %The last one follows  using that $l=2k+1+m \leq 2g-(2k+1+m)$. 
If $l$ is even (equivalently $m$  is odd), these are given by 
$$
(i,j)=(1+\lfloor \frac{m}{2}\rfloor,2k+\lfloor\frac{m}{2})\rfloor,...,(k+\lfloor\frac{m}{2}\rfloor,k+2+\lfloor\frac{m}{2}\rfloor).
$$
%Observe that we are using the hypothesis $ 2 \leq k+1 \leq \frac{g-1}{2}$ of the statment to say that $2g-2(2k+1) \geq 0 $.
This shows that $n_l \geq k$. 

On the other hand, if $l \leq 2k$ or $l \geq  2g -2k$, it is easy to see that the number of admissible indexes $(i,j)$ such that $i+j=l$ is strictly less than $k$.
%$i \leq $ %$k-1$. coordinates for $l=2g-(2k+1)+1$, or $l=2g-(2k+1)+2$, $k-2$ coordinates for $l=2g-(2k+1)+2$, or $l=2g-(2k+1)+3$...$1$ coordinate for $l=2g--4$, or $l=2g-3$.
 We then conclude that $s_k=2g-(2k+1)-2k=2g-(4k+1)$  and hence $dim(Ker(\mu_{2k})=dim Ker(\mu_{2(k-1)})-s_k=dim Ker(\mu_{2(k-1)})-(2g-(4k+1)) $.
 So we have proven that $Rank(\mu_{2k}) = 2g-(4k+1)$. 

By the equality: $dim(Ker(\mu_{2k}))=dim Ker(\mu_{2(k-1)})-(2g-(4k+1)) $, and using that $dim(I_2) = \frac{(g-1)(g-2)}{2}$, we get
$$dim(Ker(\mu_{2k})) = \frac{(g-1)(g-2)}{2}- \sum_{i=1}^k (2g-(4i+1))= \frac{(g-1)(g-2)}{2}+ k(2k-2g+3),$$
for every $k \leq \frac{g-1}{2}$. 
So if $g$ is odd and $k = \frac{g-1}{2}$, we  have $dim(Ker(\mu_{2k}))= dim(Ker(\mu_{g-1})) =0$, hence for $k>  \frac{g-1}{2}$, $\mu_{2k}\equiv 0$.
If $g$ is even, and $k = \frac{g-2}{2}$, we  have $dim(Ker(\mu_{2k}))= dim(Ker(\mu_{g-2})) =0$, hence for $k>  \frac{g-2}{2}$, $\mu_{2k}\equiv 0$, since the domain is $0$.
\\

\end{proof}

%\begin{remark}
    %Let us denote by $\omega_i=\alpha_{g-i}$, where $i=1,...,g-1$. Hence $\omega_i=x^{g-i}\frac{dx}{y}$. Notice that we have 
  %  \begin{equation}
  %      \sum \limits_{\substack{1 \leq i < j \leq g-1 \\ i+j=l}}a_{ij} (\alpha_i \wedge \alpha_j)=-\sum \limits_{\substack{1 \leq i < j \leq g-1 \\ i+j=l}}a_{ij} (\omega_{g-j} \wedge \omega_{g-i})=\sum \limits_{\substack{1 \leq k < h \leq g-1 \\ k+h=2g-l}}b_{kh} (\omega_k \wedge \omega_h)
  %  \end{equation}
   % where $b_{k,h}=-a_{g-h,g-k}$
%\end{remark}

\newpage
\section{Schiffer variations and Weierstrass points}
\label{Schiffer}
In this section we recall the definition and the basic properties of (higher) Schiffer variations. 

Let $C$ be a smooth curve of genus $g >1$, take a point $p \in C$ and fix a local coordinate $z$ centred in $p$. For $1 \leq n \leq 3g-3$, we define the $n^{th}$ Schiffer variation  at $p$ to be the element $\xi_p^n \in H^1(C, T_C) \cong H^{0,1}_{\bar{\partial}}(T_C)$ whose Dolbeault representative is given by $\frac{ \bar{\partial}\rho_p}{z^n} \frac{\partial}{\partial z}$, where $\rho_p$ is a bump function in $p$ which is equal to one in a small neighborhood $U$ containing $p$, $\xi_p^n = [\frac{ \bar{\partial}\rho_p}{z^n} \frac{\partial}{\partial z}]$. Clearly  $\xi_p^n$ depends on the choice of the local coordinate $z$. Take $1\leq n \leq 3g-3$. Consider the exact sequence 
	$$0 \rightarrow T_C  \rightarrow T_C(np)  \rightarrow T_C(np)_{|np} \rightarrow 0, $$
and the induced exact sequence in cohomology: 
$$ 0 \rightarrow H^0(T_C(np)) \rightarrow H^0(T_C(np)_{|np}) \stackrel{\delta^n_p}\rightarrow H^1(T_C). $$
By Riemann Roch, if $n <2g-2$, or $n \leq 3g-3$ and $p$ is a general point, we have: $$ h^0(T_C(np)) =0. $$ Hence we have an inclusion $$ \delta^n_p: H^0(T_C(np)_{|np})  \cong {\mathbb C}^n \hookrightarrow H^1(T_C)$$ and the image of $\delta^n_p$ in $H^1(C,T_C)$ is the $n$-dimensional subspace generated by $\xi_p^1,...,\xi_p^n$ (see \cite{fredianihigher}) for more details). 

%Clearly $H^0(K_C-np) \subset Ker (\cup \xi_p^n)$, for $n \leq g$, hence $\xi_p^n$ has rank $\leq n$. In particular any linear combination $\zeta = \sum_{i=1}^n a_i \xi_{p_i}^{m_i}$ has rank $\leq \sum_{i=1}^n m_i$, since $H^0(K_C-  \sum_{i=1}^n m_i p_i) \subset ker(\cup \zeta)$. 
Given an element $\zeta \in H^1(C, T_C)$, consider the map given by cup product: 
$$\cup \zeta: H^0(C, K_C) \rightarrow H^1(C, {\mathcal O}_C).$$
The rank of $\zeta$ is by definition the rank of the map $\cup \zeta$. 
\begin{remark}
\label{rank}

    Notice that $H^0(K_C(-np)) \subset Ker (\cup \xi_p^n)$, for $n \leq g$. 
In fact, $\omega \in ker (\cup \xi_p^n)$ if and only if for any $\alpha \in H^0(K_C),$ we have $\xi_p^n(\omega\alpha) =0$. So if $z$ is a local coordinate around $p $ and $\omega = f(z) dz$, $\alpha= g(z) dz$ are local expressions, by \cite[Lemma 2.2]{fredianihigher} we have: 
$$\xi_p^n(\omega \alpha) = \frac{2 \pi i}{(n-1)!} (fg)^{(n-1)}(p) =0,$$
if $\omega \in H^0(K_C(-np))$, $\forall \alpha \in H^0(K_C)$.
Hence $\xi_p^n$ has rank $\leq n$. 
\end{remark}

%In particular any linear combination $\zeta = \sum_{i=1}^n a_i \xi_{p_i}^{m_i}$ has rank $\leq \sum_{i=1}^n m_i$, since $H^0(K_C-  \sum_{i=1}^n m_i p_i) \subset ker(\cup \zeta)$. 

Let $\mathcal{H}_g$ be the moduli space of hyperelliptic curves of genus $g$. If $C$ is  a hyperelliptic curve, then the tangent space of $\mathcal{H}_g$ at $[C]$ is $H^1(T_C)^+$, that is the invariant subspace with respect to the hyperelliptic involution $\sigma$.\\
Assume $p \in C$ is a Weierstrass point.  From the Dolbeault representation of the Schiffer variations, one immediately sees that the elements $\xi_p^{2k+1}$ are $\sigma$-invariant, hence they belong to $H^1(T_C)^+$. So we have a subspace 
$$\langle \xi_p^1,\xi_p^3,...,\xi_p^{2g-3}\rangle \subset H^1(T_C)^+$$ of dimension $g-1$. 

Denote as usual by $|L|$ the $g^1_2$.  We conclude this section giving a basis of $H^0(K_C \otimes L^{\vee})$ that will be useful  in the next sections. 

\begin{remark}
\label{omegai}
Consider the hyperelliptic curve $C$ with equation \eqref{hypereq}. 
    Fix the Weierstrass point $p=(0,0)$. For the computations of the next section it is more convenient to use the basis of $H^0(M)$ given by $\omega_i: = \alpha_{g-i}=\frac{x^{g-i}dx}{y}$, $i =1,...,g-1$.  Notice that we have  
    $$ord_p\omega_k = 2g-2k-2, \ for  \ all  \ k=1,...,g-1.$$
     Hence we get:
                \begin{itemize}
            \item $H^0(M)= \langle \omega_1,..., \omega_{g-1}\rangle.$
            \item $H^0(M(-p))= H^0(M(-2p))=\langle \omega_1,..., \omega_{g-2}\rangle.$
            
             $\vdots$
             
            \item $H^0(M(-(2g-4)p)) = H^0(M(-(2g-5)p)) =   \langle \omega_1\rangle.$

        \end{itemize}
    
    So if we take a quadric  $Q=\sum \limits_{1 \leq i < j \leq g-1}a_{ij}(s \alpha_i \odot t \alpha_j -   s \alpha_j \odot t \alpha_i) \in I_2,$ as in Lemma \ref{lemmanahMus} or Lemma \ref{lemmasadcx}, we have: 
    $$ Q=\sum \limits_{1 \leq i < j \leq g-1}a_{ij}(s \alpha_i \odot t \alpha_j -   s \alpha_j \odot t \alpha_i)=-\sum \limits_{1 \leq k < h \leq g-1}a_{g-h, g-k}(s \omega_k \odot t \omega_h -   s \omega_h \odot t \omega_k)=$$
 $$=    \sum \limits_{1 \leq k < h \leq g-1}b_{k, h}(s \omega_k \odot t \omega_h -   s \omega_h \odot t \omega_k),$$
 with $b_{k,h} = -a_{g-h,g-k}$.

\end{remark}
\section{Second fundamental form of the Hyperelliptic Torelli locus}
\label{second}
Let $C$ be a hyperelliptic curve of genus $g \geq 3$, $L$ be the line bundle giving the $g^1_2$. Set $M = K_C\otimes L^{\vee}$ and $\pi:C \rightarrow {\mathbb P}^1$ the map induced by $|L|$.  

Denote by $\sigma$ the hyperelliptic involution and consider the decomposition    $H^0(C, K_C^{\otimes 2}) \cong H^0(C, K_C^{\otimes 2})^+ \oplus  H^0(C, K_C^{\otimes 2})^-$ of  $H^0(C, K_C^{\otimes 2})$ in invariant and anti-invariant subspaces under the action of $\sigma$.

	Denote by ${\mathcal H}_g$ the hyperelliptic locus in ${\mathcal M}_g$ and by $$j_{h}: {\mathcal H}_g \rightarrow {\mathcal A}_g$$ the restriction of the Torelli map to ${\mathcal H}_g$. We endow  ${\mathcal A}_g$ with the Siegel metric, that is the orbifold metric induced by the symmetric metric on the Siegel space $Sp(2g, {\mathbb R})/U(g)$ of which ${\mathcal A}_g$ is the quotient under the action of $Sp(2g, {\mathbb Z})$. The map $j_h$ is an orbifold immersion (see \cite{os}) and we have the following tangent bundle exact sequences

%\begin{equation}
%\label{tangentHE}
%0 \rightarrow T_{{\mathcal H}_g}\rightarrow {T_{{\mathcal A}_g}}_{|{\mathcal H}_g} \rightarrow N_{{\mathcal H}_g}/{\mathcal A}_g \rightarrow 0   
%\end{equation}

\begin{equation}
 \begin{tikzcd}
 \label{tangent}
 & & 0 \ar[d] & & &\\
    &0\ar[r] &  T_{\mathcal {H}_g}\ar[d]\ar[r] &  {T_{\mathcal{A}_g}}_{|{\mathcal {H}_g}}\ar[d, "="] \ar[r] & N_{{\mathcal{H}_g}_{/\mathcal{A}_g}} \ar[r]\ar[d] & 0\\
    &  &T_{{\mathcal{M}_g}|\mathcal{H}_g} \ar[r]& T_{{\mathcal{A}_g}|\mathcal{H}_g}\ar[r]& N_{{\mathcal{M}_g}_{/\mathcal{A}_g}|\mathcal{H}_g}\ar[r] &0\\
    & & & &  &
\end{tikzcd}
\end{equation}
Denote by 
\begin{equation}
\rho_{HE}: N^*_{\mathcal{H}_g|\mathcal{A}_g}
 \rightarrow Sym^2 \Omega^1_{\mathcal{H}_g},
\end{equation}
the dual of the second fundamental form of $j_h$. 

At a point $[C] \in {\mathcal H}_g$, the dual of \eqref{tangent} is

\begin{tikzcd}
\label{cotangent}
 % \xymatrix{
        &0 \ar[r] & I_2\ar[d] \ar[r]& S^2 H^0(K_{C})\ar[d, "="]\arrow[r,"\mu_0"]& H^0(K_{C}^{\otimes 2})\ar[d] &\\
    &0\ar[r] & I_2\ar[r] &  S^2 H^0(K_{C})\arrow[r,"\mu_0"] &  H^0(K_{C}^{\otimes 2})^+ \ar[r]& 0%}
\end{tikzcd}

%\begin{equation}
%0 \rightarrow I_2(K_C) \rightarrow S^2H^0(K_C) \rightarrow H^0(2K_C)^+ \rightarrow 0
%\end{equation}
where $m$ is the multiplication map 
and $I_2$ can be identified with the vector space of quadrics containing the rational normal curve.

So at a point $[C]$ the map $\rho_{HE}$ is a linear map 

$$I_2 \rightarrow Sym^2( H^0(C, K_C^{\otimes 2})^+).$$
In \cite[Prop. 5.1]{cftrans} it is proven that $\forall Q \in I_2$, $\forall v,w \in H^1(T_C)^+$ we have 
\begin{equation}
\label{rhoh}
\rho_{HE}(Q) (v \odot w) = \rho(Q)(v \odot w),
\end{equation}
where  $\rho$ is the Hodge Gaussian map introduced in \cite[Proposition-Definition 1.3]{cpt}. 

In \cite{fp} the second fundamental form $\rho_{HE}$ of the hyperelliptic locus has been studied using the Hodge Gaussian map $\rho$ and the second gaussian map of the canonical bundle.  In \cite[Theorem 6.2]{fp} it has been  proven that if $Y$ is a germ of totally geodesic subvariety  of ${\mathcal A}_g$ generically contained in the hyperelliptic Torelli locus, then its dimension is at most $g+1$. 

In \cite{fredianihigher} a relation between higher even
Gaussian maps of the canonical bundle on a smooth projective curve
of genus $g  \geq 4$ and the second fundamental form of the Torelli map is given. 
This is a generalisation of a result obtained by Colombo, Pirola and Tortora on
the second Gaussian map and the second fundamental form in \cite{cpt} (see also \cite[Theorem 2.2]{cfg}). 
In fact, in  \cite[Theorem 2.2]{cfg}) it is proven that for any $p \in C$, and for any $Q \in I_2$, we have 
\begin{equation}
\label{mu2}
    \rho(Q)(\xi_p \odot \xi_p) = -2 \pi i \mu_2(Q)(p).
\end{equation}

More precisely, the computations of $\rho(Q)$ on Schiffer variations given in \cite[Proposition 3.1 and Remark 3.2]{fredianihigher}, for quadrics $Q \in Ker(\mu_{2k})$ will be the main tool in the next section.

\section{Isotropic subspaces}
\label{iso}

%In this section, given a hyperelliptic curve $C$ of genus $g \geq 3$ and a Weierstrass point $p \in C$,  we will determine for any   $k \leq \lfloor\frac{g-3}{2} \rfloor$, the maximum $m \leq \lfloor \frac{g-2}{2} \rfloor$, such that the subspace $V_m = \langle \xi_p^1, \xi_p^2, ..., \xi_p^{2m+1} \rangle \subset V_{\lfloor \frac{g-2}{2} \rfloor}  \subset H^1(T_C)^+$ is isotropic with respect to all the quadrics $\rho(Q)$, $\forall Q \in Ker(\mu_{2k})$. We will show that $m = k$, namely that $V_k$ is isotropic for $\rho(Q)$, $\forall Q \in Ker(\mu_{2k})$, while $V_{k+1}$ is not (see Theorems \ref{thm1}, \ref{thm3}). 

In this section, given a hyperelliptic curve $C$ of genus $g \geq 3$ and a Weierstrass point $p \in C$,  we will determine for any   $k \leq \lfloor\frac{g-3}{2} \rfloor$, the maximum $m \leq \lfloor \frac{g-1}{2} \rfloor$, such that the subspace $V_m = \langle \xi_p^1, \xi_p^3, ..., \xi_p^{2m+1} \rangle \subset V_{\lfloor \frac{g-1}{2} \rfloor}  \subset H^1(T_C)^+$ is isotropic with respect to all the quadrics $\rho(Q)$, $\forall Q \in Ker(\mu_{2k})$. We will show that $m = k$, namely that $V_k$ is isotropic for $\rho(Q)$, $\forall Q \in Ker(\mu_{2k})$, while $V_{k+1}$ is not (see Theorems \ref{thm1}, \ref{thm3}).

In order to do this we recall that, by \eqref{rhoh}, it suffices to compute $\rho(Q) (\xi_p^r \odot \xi_p^n)$, for odd $r,n$. 

To do this we will use \cite[Remark 3.2]{fredianihigher}, that we now recall for the reader's convenience. 
\begin{remark} (\cite[Remark 3.2]{fredianihigher})
\label{remarkhigher} 
Let  
$Q=\sum \limits_{\alpha,\beta=1}^u b_{\alpha \beta} \gamma_{\alpha}\odot \gamma_{\beta}$  be a  quadric in $I_2$, where $\gamma_1,...,\gamma_u$ are elements of $H^0(K_C)$. Choose a local coordinate $z$ around $p$ and take a local expression of $\gamma_{\alpha}$ around $p$:  $\gamma_{\alpha}=g_{\alpha}(z)dz$.

Assume moreover that 
$$\sum_{\alpha, \beta} c_{\alpha \beta} g_{\alpha}^{(h)}(0) g_{\beta}^{(l)}(0)= 0, \ \forall h,l \geq 0, \ h +l \leq m,$$ 
then we have 
$$\rho_{HE}(Q)(\xi_p^n \odot \xi_p^r) = \rho(Q)(\xi_p^n \odot \xi_p^r) = 0, \ if \  r+n  \leq m, \ \forall  \ \text{odd}  \ r,n \geq 1,$$
and if $r +n = m+1$ 
$$ \rho_{HE}(Q)(\xi_p^n \odot \xi_p^r) =\rho(Q)(\xi^n_p \odot \xi^r_p) =  2 \pi i  \left( \sum_{k=0}^{n-1} \left (  \sum^u_{\alpha, \beta=1}c_{\alpha \beta} g^{(m+1-k)}_{\alpha}(0) g_{\beta}^{(k)}(0)\right)\frac{(n-k)}{k!(m+1-k)!}\right).$$

\end{remark}
  Recall that a basis of $I_2$ is given by the following quadrics (see Lemma \ref{fhgf}):  
  \begin{equation}
 \label{basis}
Q_{ij}:=s\omega_i \odot t\omega_j -s\omega_j \odot t\omega_i
 \end{equation}
for $1\leq i<j\leq g-1$, where $L$ is the $g^1_2$, $H^0(L) = \langle s, t \rangle$, and $H^0(M) = H^0(K_C \otimes L^{\vee}) = \langle \omega_1, ..., \omega_{g-1} \rangle $, where the sections $\omega_i$ are as in Remark \ref{omegai}. Let $Q$ be a quadric in $I_2$, then $$Q=\sum\limits_{1\leq i<j\leq g-1} b_{ij}Q_{ij}=\sum \limits_{\alpha,\beta=1}^u c_{\alpha \beta} \gamma_{\alpha}\odot  \gamma_{\beta},$$ where $\gamma_1,...,\gamma_u$ are elements of $H^0(K_C)$. Choose a local coordinate $z$ around $p$ and take a local expression of $\gamma_{\alpha}$ around $p$:  $\gamma_{\alpha}=g_{\alpha}(z)dz$.\\
Now let $k\geq 0$ and assume $Q\in Ker(\mu_{2k})$. Recall  from Section \ref{secidsyhbcvT} that 
$$ Q\in Ker(\mu_{2k}) \Leftrightarrow \sum \limits_{\alpha,\beta=1}^u c_{\alpha \beta} g_{\alpha}^{(h)} g_{\beta}^{(l)} \equiv 0 \quad \forall h+l\leq 2k+1.$$
Moreover by Lemma \ref{lemmasss} we have

$$ \mu_{2k+2}(Q)=\sum \limits_{\alpha,\beta=1}^u c_{\alpha \beta} g_{\alpha}^{(h)} g_{\beta}^{(l)}(dz)^{2k+4}=\mu_{1,L}(s \wedge t)\mu_{2k+1,M}(\sum_{1 \leq i <j\leq g-1} b_{ij} (\omega_i \wedge \omega_j))(dz)^{2k+4}=$$ $$=(f'g-fg')\sum \limits_{1\leq i<j\leq g-1}b_{ij}(f_i^{(r)}f_j^{(s)}-f_i^{(s)}f_j^{(r))}) (dz)^{2k+4}\quad \forall h+l=2k+2, \quad \forall r+s=2k+1,$$
where in a local coordinate we have $s = f(z) T$, $t = g(z) T$,  $\omega_i = f_i(z) dz \otimes T^{\vee} $, $\forall i$. 

It follows that $$\sum \limits_{\alpha,\beta=1}^u c_{\alpha \beta} g_{\alpha}^{(h)} g_{\beta}^{(l)}(p)=0 \quad \forall h+l=2k+2,$$ because $(f'g-fg')(p)=0$, since $p$ is a Weierstrass point.\\

Now we want to determine the least $m \geq 2k+3$ such that 
%So a natural question arises: which is the first order of derivatives $m$ (depending on $k$) such that 

$$\sum \limits_{\alpha,\beta=1}^u c_{\alpha \beta} g_{\alpha}^{(h)} g_{\beta}^{(l)}(p)\neq 0, $$
for some $h,l$ with $h+l=m.$ 

%Notice that from what we have already said, we know that $m\geq 2k+3$, but using our particular basis we will see that it is actually bigger.

\begin{lemma}
\label{lemma1}
    Let $Q=\sum \limits_{1\leq r<m\leq g-1} b_{r,m}Q_{r,m} \in I_2$, where the quadrics $Q_{r,m}$ are considered with respect to the basis $\{\omega
    _i\}_{i=1,...,g-1}$ given in Remark \ref{omegai}. If $Q\in Ker(\mu_{2k})$, we have
    $$
    b_{r,m}=0 \quad \forall \quad  r+m\geq 2g-(2k+2) \qquad (hence \ m\geq g-k).
    $$
\end{lemma}
\begin{proof}
Recall that we have $Q=\sum \limits_{1\leq i<j\leq g-1} a_{ij}(s\alpha_i \odot t\alpha_j - s\alpha_j \odot t\alpha_i)$, where $a_{ij}=-b_{g-j,g-i}$ (see Remark \ref{omegai}). We will prove by induction on $k$ that
\begin{equation}
\label{coeff}
    a_{ij}=0 \quad \forall \quad i+j\leq 2k+2 .
\end{equation}
Notice that this finishes the proof. Indeed,  putting $r=g-j, m=g-i$ we have $$i+j\leq 2k+2 \Leftrightarrow r+m=2g-(i+j)\geq 2g-(2k+2).$$ 
We also get $m\geq g-k$, because  if $m\leq g-k-1$,  $r<m$ implies $r+m< 2g-2k-2$. \\
Recall that in Lemma \ref{lemmanahMus} we have shown that $Q \in Ker(\mu_{2})$ if and only if 
for every $3 \leq l \leq 2g-3$, 
\begin{equation}
\label{mu2}
\sum \limits_{\substack{1 \leq i < j \\ i+j=l}} a_{ij} (j-i)=0.
\end{equation}
 By Lemma \ref{lemmasadcx}, $Q \in Ker(\mu_{2k})$, $k \geq 2$, if and only if $Q \in Ker(\mu_{2})$ and and for every $m,l$ such that  $2 \leq m \leq k$, $  2m-1 \leq l \leq 2g-3$,
\begin{equation}
\label{mu2k}
\sum \limits_{\substack{1 \leq i < j \\ i+j=l \\ i \geq m-1, \\ j \geq m-1 }} a_{ij} (j-i) ij(i-1)(j-1)...(i-(m-2))(j-(m-2))=0.
\end{equation}
So let us prove \eqref{coeff}. The case $k=1$ is trivial, since $2k+2=4$ and we only have $l=3,4$. From equations \eqref{mu2} we have $a_{12}=a_{13}=0$. Now assume that the result is true for $k-1$, where $k\geq 2$. Hence if $Q\in Ker(\mu_{2(k-1)})$, we have $a_{ij}=0 \quad  \forall l=i+j\leq 2k $. Now if $Q\in Ker(\mu_{2k}) $ then  $Q\in Ker(\mu_{2k-2})$, hence $a_{ij}=0 \quad \forall l\leq 2k$. So we only have to check the cases $l=2k+1,2k+2$.
\begin{itemize}
    \item If $l=2k+1$, then we claim that $i\leq k$. Otherwise, if  $i\geq k+1$, we have $j>i\geq k+1$, hence  $i+j\geq 2k+3$, a contradiction. In the equations \eqref{mu2k} with $l=2k+1$, we have  $a_{1,2k},a_{2,2k-1},...,a_{k,k+1}$ as variables, since $i \leq k$. The  equations  given by \eqref{mu2} and \eqref{mu2k} for $m=2,...,k$, $l=2k+1$,  are then a system of  $k$ equations in $k$ variables, which are  linearly independent by the proof of Theorem \ref{rteoremainrese}. So we get $a_{i,j}=0$ for $i+j=2k+1$.
    \item If $l=2k+2$, we get $i\leq k$ exactly as in the previous case. Hence again in  \eqref{mu2k} with $l=2k+2$, we have exactly $k$ linearly independent equations in the $k$ variables $a_{1,2k+1},a_{2,2k},...,a_{k,k+2}$. So $a_{ij}=0$ if $i+j=2k+2$.
\end{itemize}
\end{proof}
\begin{remark} 
\label{remark1}
    Let $\omega_1,...,\omega_{g-1}$ be our basis of $H^0(M)$ and let $z$ be a local coordinate centered in the Weierstrass point $p$. Assume that we locally have $\omega_i=f_i(z)dz\otimes T^{\vee}$, by abuse of notation we will write $\omega_i^{(h)}(p)$ instead of $f_i^{(h)}(0)$. Then by Remark \ref{omegai} we have $$\omega_i^{(h)}(p)=0 \quad \forall h\leq 2g-(2i+3).$$ 
    \end{remark}
\begin{lemma}
    \label{lemma2}
    %Let $C$ be a hyperelliptic curve and $p$ a Weierstrass point. Let $\omega_1,...,\omega_{g-1}$ be  basis of $H^0(K_C-L)$. Then we have
    With the above notation, we have 
    $$
    \sum \limits_{1\leq i<j\leq g-1} b_{ij}(\omega_i^{(h)}\omega_j^{(l)}-\omega_i^{(l)}\omega_j^{(h)})(p)=0 \quad \forall h+l\leq 4k+1.
    $$
\end{lemma}
\begin{proof}
    We may assume $h>l$, since if $h=l$ we have trivially zero. From Remark \ref{remark1}, $\omega_j^{(l)}(p)=0 \quad \forall j$ such that $2j+3\leq 2g-l$. Moreover, since $i<j$, this also implies $\omega_i^{(l)}(p)=0$, so the terms in the sum with  $2j+3\leq 2g-l$ are identically zero. Then we can assume that the  indexes $j$ that appear in the sum are those satisfying  $2j+3\geq 2g-l+1$, which gives $j \geq g-\frac{l+2}{2} $. 
%\begin{center}
 %   $j\geq g-\frac{l+4}{2} \quad$  if $l$ is even \\
  %      $j\geq g-\frac{l+5}{2}\quad $ if $l$ is odd.
%\end{center}
Denote by
$\Tilde{l}:= \frac{l+2}{2}$. 
Hence we have $$\sum \limits_{1\leq i<j\leq g-1} b_{ij}(\omega_i^{(h)}\omega_j^{(l)}-\omega_i^{(l)}\omega_j^{(h)})(p)=\sum \limits_{1\leq i<j\leq g-1,j\geq g-\Tilde{l}} b_{ij}(\omega_i^{(h)}\omega_j^{(l)}-\omega_i^{(l)}\omega_j^{(h)})(p)=$$ 
\begin{multline*}
    =\sum \limits_{1\leq i< g-1} b_{i,g-1}(\omega_i^{(h)}\omega_{g-1}^{(l)}-\omega_i^{(l)}\omega_{g-1}^{(h)})(p)+\sum \limits_{1\leq i< g-2} b_{i,g-2}(\omega_i^{(h)}\omega_{g-2}^{(l)}-\omega_i^{(l)}\omega_{g-2}^{(h)})(p) \\+ \dots +\sum \limits_{1\leq i<g-\Tilde{l}} b_{i,g-\Tilde{l}}(\omega_i^{(h)}\omega_{g-\Tilde{l}}^{(l)}-\omega_i^{(l)}\omega_{g-\Tilde{l}}^{(h)})(p).
\end{multline*}
Let us focus on the first term of the above sum. From \Cref{lemma1}, $Q\in Ker(\mu_{2k})$ implies $b_{i,g-1}=0\quad \forall g-(2k+1)\leq i\leq g-2$. Then we have 

\begin{equation}
\label{first}
\sum \limits_{1\leq i< g-1} b_{i,g-1}(\omega_i^{(h)}\omega_{g-1}^{(l)}-\omega_i^{(l)}\omega_{g-1}^{(h)})(p)=\sum \limits_{i=1}^{g-2k-2} b_{i,g-1}(\omega_i^{(h)}\omega_{g-1}^{(l)}-\omega_i^{(l)}\omega_{g-1}^{(h)})(p).
\end{equation}
Moreover, again by \Cref{remark1}, if $2i+3\leq 2g-h<2g-l$, then $\omega_i^{(h)}(p)=\omega_i^{(l)}(p)=0$. But in the above sum we have $2i+3\leq 2(g-2k-2)+3=2g-4k-1$. 
Hence, if $h\leq 4k+1$  (so $l < 4k+1$), then $2g-4k-1\leq 2g-h<2g-l$, so $2i+3 \leq  2g-h<2g-l$ in all summands in \eqref{first}, so the sum is identically zero.

Let us now consider the general term, where $j=g-s$. From \Cref{lemma1}, we have $b_{i,g-s}=0$ $\forall g-(2k+2-s)\leq i < g-s$. Hence the $i$'s that actually appear in the sum are those satisfying $i\leq g-(2k+2-s)-1=g-2k-3+s$. In the same way as we did for the first term, for $h\leq 4k+3-2s $, we have $2i+3\leq 2(g-2k-3+s)\leq 2g-h$, so 
$\omega_i^{(h)}(p)=\omega_i^{(l)}(p)=0$ and the term is zero. 

 It remains to look at the cases when $h\geq 4k+4-2s$. Since $h+l\leq 4k+1$ by assumption, we must have $4k+4-2s\leq h+l\leq 4k+1$, which gives $l\leq 2s-3$. For these values of $l$, we have $\omega_{g-s}^{(l)}(p)=0$ by Remark \ref{remark1}. Hence the term is zero again, because $i<j$, so $\omega_i^{(l)}(p)=0$. Thus we have shown that all terms are zero. 
 \end{proof}
We are now ready to state and prove our result.
\begin{theorem}
\label{thm1}
    Let $C$ be an hyperelliptic curve of genus $g$, $p\in C$ a Weierstrass point. Let $\gamma_1,...,\gamma_r \in H^0(K_C)$ where locally $\gamma_{\alpha}=g_{\alpha}dz$ and let $$Q=\sum \limits_{\alpha,\beta=1}^u c_{\alpha \beta} \gamma_{\alpha}\gamma_{\beta}=\sum \limits_{1\leq i<j\leq g-1} b_{ij}Q_{ij} \in Ker(\mu_{2k}).$$ Then $$
    \sum \limits_{\alpha,\beta=1}^u c_{\alpha \beta} g_{\alpha}^{(h)} g_{\beta}^{(l)}(p)=0 \quad \forall h+l\leq 4k+3.
    $$
\end{theorem}
\begin{proof}
We have to compute the following expression $\forall h+l\leq 4k+3$:
\begin{multline}
    \label{eq1}
    \sum \limits_{\alpha,\beta=1}^u c_{\alpha \beta} g_{\alpha}^{(h)} g_{\beta}^{(l)}(p)=\\=\sum \limits_{1\leq i<j\leq g-1} b_{ij}[(s\omega_i)^{(h)}(t\omega_j)^{(l)}+(t\omega_j)^{(h)}(s\omega_i)^{(l)}-(s\omega_j)^{(h)}(t\omega_i)^{(l)}-(t\omega_i)^{(h)}(s\omega_j)^{(l)}](p).
\end{multline}
We have
\begin{equation*}
    (s\omega_i)^{(h)}(t\omega_j)^{(l)}-(s\omega_j)^{(h)}(t\omega_i)^{(l)}
    =\sum\limits_{n=0}^h\sum\limits_{m=0}^l \binom{h}{n} \binom{l}{m} s^{(h-n)}t^{(l-m)}(\omega_i^{(n)}\omega_j^{(m)}-\omega_j^{(n)}\omega_i^{(m)}).
\end{equation*}

Hence the term inside the square brackets of \Cref{eq1} becomes 
\begin{equation*}
   \sum\limits_{n=0}^h\sum\limits_{m=0}^l \binom{h}{n} \binom{l}{m} (\omega_i^{(n)}\omega_j^{(m)}-\omega_j^{(n)}\omega_i^{(m)})(s^{(h-n)}t^{(l-m)}-t^{(h-n)}s^{(l-m)}).
\end{equation*}
So equation  \eqref{eq1} is
\begin{equation}
\label{eq2}
    \sum\limits_{n=0}^h\sum\limits_{m=0}^l  \binom{h}{n} \binom{l}{m}(s^{(h-n)}t^{(l-m)}-t^{(h-n)}s^{(l-m)})(p)\sum \limits_{1\leq i<j\leq g-1} b_{ij}(\omega_i^{(n)}\omega_j^{(m)}-\omega_j^{(n)}\omega_i^{(m)})(p).
\end{equation}
From \Cref{lemma2} we have that
$$
\sum \limits_{1\leq i<j\leq g-1} b_{ij}(\omega_i^{(n)}\omega_j^{(m)}-\omega_j^{(n)}\omega_i^{(m)})(p)=0\quad \forall m+n\leq 4k+1.
$$
Since $n\leq h,m\leq l$ we have $m+n\leq h+l$, so this implies that the expression in  \eqref{eq2} is zero for all $h,l$ such that $h+l \leq 4k+1$.\\
Now assume $h+l=4k+2$. If $n+m\leq 4k+1$ the expression is zero exactly as before, so the only case to consider is $n+m=4k+2$ and we must have $n=h,m=l$. But in this case we have trivially zero, because $(st-ts)=0$.\\
We are only left to the case $h+l=4k+3$. Now again, if $m+n\leq 4k+1$ we get zero and if $m+n=4k+3$ we get zero as in the previous case (this would imply again $n=h,m=l$). So we can only have $n+m=4k+2$. This implies either $n=h-1,m=l$ or $n=h,m=l-1$. 
In both cases, the first factor of \eqref{eq2} is (up to scalar) $(s't-t's)(p)=0$
 since $p$ is a Weierstrass point. Hence we have proven that \eqref{eq2} is zero $\forall h+l\leq 4k+3$ and this concludes the proof.

\end{proof}
For any $k \leq 2g-3$, consider the $(k+1)$-dimensional subspace $V_k:=\langle \xi_p^1,\xi_p^3,...,\xi_p^{2k+1}\rangle \subset H^1(T_C)^+$. We have the following 
\begin{theorem}
\label{thm2}
    Let $0\leq k\leq  \lfloor \frac{g-3}{2} \rfloor$ and let $Q\in Ker(\mu_{2k})$. Then we have 
    \begin{equation*}
   \rho_{HE}(Q)(\xi_p^l\odot \xi_p^m)=     \rho(Q)(\xi_p^l\odot \xi_p^m)=0 \quad \forall l+m\leq 4k+3.
    \end{equation*}
    This implies that the subspace 
    $V_k=\langle \xi_p^1,\xi_p^3,...,\xi_p^{2k+1}\rangle \subset H^1(T_C)^+$
     is isotropic for $\rho_{HE}(Q)$,  $\forall Q\in Ker(\mu_{2k}).$
\end{theorem}
\begin{proof}
    The second assertion follows directly from the first one. The first one follows immediately by Remark \ref{remarkhigher} (\cite[Remark 3.2]{fredianihigher}) and by Theorem \ref{thm1}. 
\end{proof}
We have the following theorem.
\begin{theorem}
\label{thm3}
    Assume $0\leq k\leq \lfloor \frac{g-3}{2} \rfloor$. Then there exists $Q \in Ker(\mu_{2k})$ such that $\rho(Q) ( \xi_p^{2k+1} \odot\xi_p^{2k +3}) \neq 0$. Hence the subspace $V_{k+1}=\langle \xi_p^1,...,\xi_p^{2k+1},\xi_p^{2k+3}\rangle  \subset H^1(T_C)^+$ is not isotropic for $\rho_{HE}(Ker(\mu_{2k}))$. \end{theorem}
\begin{proof}
    Let $Q=\sum\limits_{\alpha,\beta=1}^rc_{\alpha,\beta}\gamma_{\alpha} \gamma_{\beta}$, where $\gamma_1,...,\gamma_r \in H^0(K_C)$. We will prove that there exists $Q$ such that $\rho(Q)(\xi_p^{2k+3}\odot\xi_p^{2k+1})\neq 0$. By Remark \ref{remarkhigher}, we know that 
    \begin{equation}
        \rho(Q)(\xi_p^{2k+3}\odot\xi_p^{2k+1})=2\pi i\sum\limits_{u=0}^{2k}\left[\sum\limits_{\alpha,\beta=1}^r c_{\alpha,\beta}\gamma_{\alpha}^{(4k+4-u)}(0)\gamma_{\beta}^{(u)}(0)\right]\frac{(2k+1-u)}{u!(4k+4-u)!}
    \end{equation}
      Let us focus our attention on  the term inside the square brackets, which is
      \begin{equation}
      \label{equationstar}
          \sum\limits_{\alpha,\beta=1}^r c_{\alpha,\beta}\gamma_{\alpha}^{(4k+4-u)}(0)\gamma_{\beta}^{(u)}(0). \quad 
      \end{equation}
      We see that we must compute the derivatives of order $(h,l)$ of $Q$ where $h+l=4k+4$ and $0\leq l\leq 2k$. In order to do this, we will again use our basis of $H^0(K_C\otimes L^{\vee})$ given by the $\omega_i$'s, exactly as in \eqref{eq1} and \eqref{eq2}. Hence we must compute 
    \begin{equation*}
        \sum\limits_{n=0}^h\sum\limits_{m=0}^l  \binom{h}{n} \binom{l}{m}(s^{(h-n)}t^{(l-m)}-t^{(h-n)}s^{(l-m)})(p)\sum \limits_{1\leq i<j\leq g-1} b_{ij}(\omega_i^{(n)}\omega_j^{(m)}-\omega_j^{(n)}\omega_i^{(m)})(p)
    \end{equation*}
    where $h+l=4k+4$. If $n+m\leq 4k+1$, we know that it is zero by \Cref{lemma2}. If $n+m=4k+4$, then in the equation above, we have $n=h$, $m=l$, so the first factor is (up to scalar) $(st-ts)=0$. If $n+m=4k+3$, we have either $n=h-1$ and $m=l$ or $n=h$ and $m=l-1$. In both cases, the first factor is (up to scalar) $(s't-t's)(p)=0$, because $p$ is a Weierstrass point. So we can only have $n+m=4k+2$, where $h=n$ and $m=l-2$ or $h=n-2$ and $m=l$, since if $n=h-1$ and $m=l-1$, the first factor is $(s't'-t's')=0$.\\
    Summarising, in the above equation we only have to consider the following terms:
    \begin{enumerate}
        \item $n=h$ and $m=l-2$, which only works if $l\geq 2$ and gives 
        \begin{equation*}
            (st''-ts'')(p)\sum \limits_{1\leq i<j\leq g-1} b_{ij}(\omega_i^{(h)}\omega_j^{(l-2)}-\omega_j^{(h)}\omega_i^{(l-2)})(p).
        \end{equation*}
        \item $n=h-2$ and $m=l$, which gives 
        \begin{equation*}
            (s''t-t''s)(p)\sum \limits_{1\leq i<j\leq g-1} b_{ij}(\omega_i^{(h-2)}\omega_j^{(l)}-\omega_j^{(h-2)}\omega_i^{(l)})(p).
            \end{equation*}
    \end{enumerate}

Let us consider the first case. By Remark \ref{remark1}, we have $\omega_j^{(l-2)}(p)=0$ if $2j+3\leq 2g-l+2$, which gives $2j\leq 2g-l-1$. We have two subcases here: $l$ odd or $l$ even.
\begin{itemize}
    \item Assume $l$ is odd. We claim that we get zero. In fact, if $j\leq g-\frac{l+1}{2}$, then $\omega_j^{(l-2)}(p)=0$ by Remark \ref{remark1}. Hence we can assume $j\geq g-\frac{l+1}{2}+1=g-\frac{l-1}{2}$. Let $j:=g-u$, hence $u\in \{1,...,\frac{l-1}{2}\}$. From \Cref{lemma1} we have $b_{i,g-u}=0$ for $i\geq g-2k-2+u$, so we may assume $i\leq g-2k-3+u$.

    We know that $\omega_i^{(l-2)}(p)=0$ if $i\leq g-\frac{l+1}{2}$, but $i\leq g-2k-3+u\leq g-\frac{l+1}{2}$ if and only if $u\leq 3+2k-\frac{l+1}{2}$. Observe that this is true since $u\leq \frac{l-1}{2}$ and $\frac{l-1}{2}\leq 3+2k-\frac{l+1}{2}$, since $l\leq 2k$. Hence $\omega_i^{(l-2)}(p)=0$.\\
    Now let us consider $\omega_i^{(h)}(p)$. We know that $\omega_i^{(h)}(p)=0$ if $2i\leq 2g-h-3$, but $i\leq g-2k-3+u$ implies $2i\leq 2g-4k-6+2u$. Now, $2g-4k-6+2u\leq 2g-h-3$ if and only if $2u\leq 4k+3-h$, which gives $l-1\leq 4k+3-h$, which is true if $l+h=4k+4$, that is  our assumption. Hence $\omega_i^{(h)}(p)=0$, for all $i$ and hence every term vanishes.
    \item Now assume $l$ is even. If $j\leq g-\frac{l+1}{2}=g-\frac{l}{2}-\frac{1}{2}<g-\frac{l}{2}$, we get $\omega_j^{(l-2)}(p)=0$ by Remark \ref{remark1}. So we may assume $j\geq g-\frac{l}{2}$. Setting $j=g-u$, we get $u\leq \frac{l}{2}$. As in the odd case, we may assume $i\leq g-2k-3+u$. Now $\omega_i^{(l-2)}(p)=0$ if $i\leq g-\frac{l}{2}-1$ and $g-2k-3+u\leq g-\frac{l}{2}-1$ if and only if $u\leq 2k+2-\frac{l}{2}$. Since $u\leq \frac{l}{2}$, we get that the last inequality is true if and only if $l\leq 2k+2$, which is the case because $l\leq 2k$. Hence $\omega_i^{(l-2)}(p)=0$.\\
    We have that $\omega_i^{(h)}(p)=0$ if $i\leq g-\frac{h}{2}-\frac{3}{2}<g-\frac{h}{2}-1$, so $\omega_i^{(h)}(p)=0$ if $i\leq g-\frac{h}{2}-2$. Hence we can assume $i\geq g-\frac{h}{2}-1$, but $i\leq g-2k-3+u$ and $g-2k-3+u-(g-\frac{h}{2}-1)=u-2-2k+\frac{h}{2}\leq \frac{l}{2}+\frac{h}{2}-2-2k=0$ because $h+l=4k+4$. Hence if $u=l/2$, there is a unique nonzero term, which is precisely $b_{g-\frac{h}{2}-1,g-\frac{l}{2}}\omega_{g-\frac{h}{2}-1}^{(h)}(p)\omega_{g-\frac{l}{2}}^{(l-2)}(p)$.
\end{itemize}
Let us now consider the second case. Again, we should distinguish the cases $l$ odd and $l$ even. However, arguing exactly in the same way as in the previous case, we again get that if $l$ is odd every term is zero.

So let us assume $l$ even, then also $h$ is even.  By Remark \ref{remark1}, $\omega_j^{(l-2)(p)}=0$ if $j\leq g-\frac{l}{2}-\frac{3}{2}<g-\frac{l}{2}-1$, so we may assume $j\geq g-\frac{l}{2}-1$. As usual we put $j=g-u$ and we have $u\leq \frac{l}{2}+1$. We have $\omega_i^{(l)}(p)=0$ if $2i\leq 2g-l-3$, but, as in the previous case, we can assume $i\leq g-2k-3+u$. 
Since $l \leq 2k$, and $u\leq \frac{l}{2}+1$, we have $2g-6-4k+2u\leq 2g-l-3$. So $2i \leq 2g-6-4k+2u \leq 2g-l-3$, and hence $\omega_i^{(l)}(p)=0$.

%We have $2g-6-4k+2u\leq 2g-l-3$ if and only if $2u\leq 4k+3-l$. Since $u\leq \frac{l}{2}+1$, the last inequality is equivalent to $l+2\leq 4k+3-l$, which gives $l\leq 2k +\frac{1}{2}$ and so this is true. 

%Hence $\omega_i^{(l)}(p)=0$. \\

Again, by Remark \ref{remark1}, $\omega_i^{(h-2)}(p)=0$ if $2i\leq 2g-h-1$. As usual we may assume $2i\leq 2g-6-4k+2u$. 
We have $u \leq \frac{l}{2}+1$. We claim that the only nonzero term occurs for $u = \frac{l}{2}+1$ and $i=g-\frac{h}{2}$. In fact, if $u \leq \frac{l}{2}$, we have $2g-6-4k+2u \leq 2g-h-1$, so $\omega_i^{(h-2)}(p)=0$. So $u = \frac{l}{2}+1$,  $i = g-\frac{h}{2}$ and the only term is $b_{g-\frac{h}{2},g-\frac{l}{2}-1}\omega_{g-\frac{h}{2}}^{(h-2)}(p)\omega_{g-\frac{l}{2}-1}^{(l)}(p)$

We have computed all the non-zero terms of       \eqref{equationstar}. These are 
\begin{itemize}
    \item $(st''-ts'')(p)[\binom{l}{2}b_{g-\frac{h}{2}-1,g-\frac{l}{2}}\omega_{g-\frac{h}{2}-1}^{(h)}(p)\omega_{g-\frac{l}{2}}^{(l-2)}(p)-\binom{h}{2}b_{g-\frac{h}{2},g-\frac{l}{2}-1}\omega_{g-\frac{h}{2}}^{(h-2)}(p)\omega_{g-\frac{l}{2}-1}^{(l)}(p)]$ if $l\geq 2$ and $h,l$ even;
    \item $-(st''-ts'')(p)\binom{h}{2}b_{g-\frac{h}{2},g-\frac{l}{2}-1}\omega_{g-\frac{h}{2}}^{(h-2)}(p)\omega_{g-\frac{l}{2}-1}^{(l)}(p)$ if  $l=0$ and $h=4k+4$.
\end{itemize}
We are now ready to complete the proof. Since $h+l=4k+4$, this implies $\frac{h}{2}=2k+2-\frac{l}{2}$, so that $g-\frac{h}{2}=g-2-2k+\frac{l}{2}$. We have
\begin{multline*}
    \rho(Q)(\xi_p^{2k+3},\xi_p^{2k+1})= 
    2\pi i\sum_{\substack{ 2 \leq l \leq 2k \\ l\text{ even}}}\frac{(2k+1-l)}{l!(4k+4-l)!} (st''-ts'')(p) \cdot \\
   \left[ \binom{l}{2}b_{g-3-2k+\frac{l}{2},g-\frac{l}{2}}\omega_{g-3-2k+\frac{l}{2}}^{(4k+4-l)}(p)\omega_{g-\frac{l}{2}}^{(l-2)}(p)- 
   \binom{h}{2}b_{g-2-2k+\frac{l}{2},g-\frac{l}{2}-1}\omega_{g-2-2k+\frac{l}{2}}^{(4k+2-l)}(p)\omega_{g-\frac{l}{2}-1}^{(l)}(p) \right] - \\ -2\pi i (st''-ts'')(p)\frac{2k+1}{(4k+4)!}\binom{4k+4}{2}b_{g-2k-2,g-1}\omega_{g-2k-2}^{(4k+2)}(p)\omega_{g-1}(p)
\end{multline*}
Notice that every term in the equation is a linear combination of $b_{r,u}$ where $r+u=2g-2k-3$.\\
Let $\lambda_{k,u}$ be the coefficient of $b_{g-2k-3+u,g-u}$ where $u=1,...,k+1$. We will now see that $\lambda_{k,u}\neq 0$ for every $k,u$. 
%Let us consider the case $u=1$:
%\begin{multline*}
%    \lambda_{k,1}=-\frac{2k+1}{(4k+4)!}\binom{4k+4}{2}\omega_{g-2k-2}^{(4k+2)}(p)\omega_{g-1}(p)+\frac{2k-1}{(4k+2)!}\omega_{g-2k-2}^{(4k+2)}(p)\omega_{g-1}(p)=\\=\omega_{g-2k-2}^{(4k+2)}(p)\omega_{g-1}(p)[-\frac{2k+1}{2(4k+2)!}+\frac{2k-1}{(4k+2)!}]=\omega_{g-2k-2}^{(4k+2)}(p)\omega_{g-1}(p)\frac{2k-3}{2(4k+2)!}
%\end{multline*}
For $1\leq u\leq k$, we have
\begin{multline*}
    \lambda_{k,u}=\omega_{g-3-2k+u}^{(4k+4-2u)}(p)\omega_{g-u}^{(2u-2)}(p)\left[\frac{2k+1-2u}{(4k+4-2u)!}\binom{2u}{2}-\binom{4k+6-2u}{2}\frac{2k+3-2u}{(4k+6-2u)!}\right]=\\=\omega_{g-3-2k+u}^{(4k+4-2u)}(p)\omega_{g-u}^{(2u-2)}(p)\frac{-8u^3+8u^2(k+1)-4ku-2k-3}{2(4k+4-2u)!}
\end{multline*}
where the last equality follows after some easy computations. Notice that 
\begin{equation*}
    -8u^3+8u^2(k+1)-4ku-2k-3\neq 0
\end{equation*}
because it is odd for every $k,u$. Moreover, $\omega_{g-3-2k+u}^{(4k+4-2u)}(p)\omega_{g-u}^{(2u-2)}(p)\neq 0$ because $\omega_{g-3-2k+u}^{(4k+4-2u)}(p)$ has order $4k+4-2u$ at $p$ and $\omega_{g-u}^{(2u-2)}$ has order  $2u-2$ at $p$ (see Lemma \ref{omegai}). \\
Finally, if $u=k+1$ we have $$\lambda_{k,k+1}=-\frac{1}{2(2k+2)!}\omega_{g-2-k}^{(2k+2)}(p)\omega_{g-k-1}^{(2k)}.$$ Hence $\lambda_{k,u}\neq 0$, for all $u =1,...,k+1$. 

So assume by contradiction that 

\begin{equation}
\label{2k+3}
        \rho(Q)(\xi_p^{2k+3}\odot \xi_p^{2k+1})=\sum\limits_{u=1}^{k+1} \lambda_{k,u}b_{g-2k-3+u,g-u}=- \sum\limits_{u=1}^{k+1} \lambda_{k,u}a_{u,2k+3-u}=0
    \end{equation}
for every $Q \in Ker(\mu_{2k})$, where $a_{ij}=-b_{g-j,g-i}$ by Remark \ref{omegai}.  Recall that $Q \in Ker(\mu_{2k})$ if and only if the coefficients $a_{ij}$ satisfy the equations given in Lemma \ref{lemmasadcx}. Moreover observe that from Theorem \ref{rteoremainrese} it follows that we have   $k$ linearly independent equations in the   variables $a_{u,2k+3-u}$ with $u=1,...,k+1$. We have shown  that each coefficient $\lambda_{k,u}$ in equation \eqref{2k+3} is different from zero and up to scalar it is the product  $\omega_{g-3-2k+u}^{(4k+4-2u)}(p)\omega_{g-u}^{(2u-2)}(p)$. So we can multiply the sections $\omega_{g-u}$ by a non zero scalar multiple, hence equation \eqref{2k+3} cannot be linearly dependent from the $k$ equations giving $Ker(\mu_{2k})$. Hence there exists a quadric $Q \in ker(\mu_{2k})$ such that  $\rho(Q)(\xi_p^{2k+3}\odot \xi_p^{2k+1}) \neq 0$.

\end{proof}
\begin{remark}

Notice that we have shown that the equation \eqref{2k+3}  gives a hyperplane in $Ker(\mu_{2k})$ that we denote by 
\begin{equation*}
A_{k,0}:=\{Q\in Ker(\mu_{2k}) : \rho(Q)(\xi_p^{2k+3}\odot\xi_p^{2k+1})=0\}.
\end{equation*}
Moreover we have just proven that a quadric $Q\in A_{k,0}$ iff satisfies the system of $k+1$ linearly independent equations in $k+1$ variables, given by the equations in Lemma \ref{lemmasadcx} defining $Ker(\mu_{2k})$ and equation \eqref{2k+3}. In particular the coefficients of $Q$ satisfy the equations $ b_{g-2k+3+u,g-u}=-a_{u,2k+3-u}=0 \quad  \forall u=1,...,k+1$.
\end{remark}
\begin{example}
    If $k=0$ we have a very concrete description of $A_{0,0}$. In fact, we have $\rho(Q_{ij})(\xi_p,\xi_p^3)=0$, $\forall (i,j)\neq (g-2,g-1)$ and $\rho(Q_{g-2,g-1})(\xi_p,\xi_p^3)\neq 0$. Hence we get that $A_{0,0}=\langle Q_{i,j}\mid 1\leq i<j\leq g-1, (i,j)\neq (g-2,g-1)\rangle$.
\end{example}
\begin{theorem}
\label{Ak0}
    Let $0\leq k\leq \lfloor \frac{g-3}{2} \rfloor$ and let us now assume $Q\in A_{k,0}$. Then there exist coefficients $\alpha_{k,u}$, all non-zero such that
    $$
    \rho(Q)(\xi_p^{2k+3}\odot \xi_p^{2k+3})=\sum\limits_{u=1}^{k+1}\alpha_{k,u}b_{g-4-2k+u,g-u}\omega_{g-4-2k+u}
^{(4k+6-2u)}(p)\omega_{g-u}^{(2u-2)}(p),$$ 
Moreover, $$ \rho(Q)(\xi_p^{2k+3}\odot \xi_p^{2k+3})=0 \Leftrightarrow b_{g-4-2k+u,g-u} =0\quad \forall u=1,...,k+1.$$
\end{theorem}
\begin{proof}
    The proof is very similar to the one given in \Cref{thm3}, where by Remark \ref{remarkhigher} we must now compute 
    $$
    \rho(Q)(\xi_p^{2k+3}\odot \xi_p^{2k+3})=2\pi i[\sum\limits_{u=0}^{2k+2}\sum\limits_{\alpha,\beta=1}^r c_{\alpha,\beta}\gamma_{\alpha}^{(4k+6-u)}(0)\gamma_{\beta}^{(s)}(0)]\frac{(2k+3-u)}{u!(4k+6-u)!}.
    $$
    By looking at the term inside the square brackets, we see that we must compute the derivatives of order $(h,l)$ of $Q$ where $h+l=4k+6$ and $0\leq l\leq 2k+2$. Hence we have to compute
    
    \begin{equation}
    \label{eq3}
    \sum\limits_{n=0}^h\sum\limits_{m=0}^l  \binom{h}{n} \binom{l}{m}(s^{(h-n)}t^{(l-m)}-t^{(h-n)}s^{(l-m)})(p)\sum \limits_{1\leq i<j\leq g-1} b_{ij}(\omega_i^{(n)}\omega_j^{(m)}-\omega_j^{(n)}\omega_i^{(m)})(p)
    \end{equation}
    where $h+l=4k+6$.\\
    The first step is analogous to Lemma \ref{lemma2}: if $Q\in A_{k,0}$ we have 
    $$
    \sum \limits_{1\leq i<j\leq g-1} b_{ij}(\omega_i^{(n)}\omega_j^{(m)}-\omega_i^{(m)}\omega_j^{(n)})(p)=0 \quad \forall m+n\leq 4k+3.
    $$
    Then one immediately sees that in \eqref{eq3} we only have to consider the case $n+m=4k+4$. Now, after some computations which are analogous to the ones in \Cref{thm3}, we get 
    $$
    \rho(Q)(\xi_p^{2k+3}\odot \xi_p^{2k+3})=\sum\limits_{u=1}^{k+1}\alpha_{k,u}b_{g-4-2k+u,g-u}\omega_{g-4-2k+u}
^{(4k+6-2u)}(p)\omega_{g-u}^{(2u-2)}(p),$$
where $\alpha_{k,u}\neq 0$ for every $u$. Then we conclude as in \Cref{thm3} that $$\rho(Q)(\xi_p^{2k+3}\odot \xi_p^{2k+3})=0 \Leftrightarrow b_{g-4-2k+u,g-u} =0\quad \forall u=1,...,k+1. $$
\end{proof}
\begin{remark}
    Let $Q\in A_{k,0}$. By Theorem \ref{Ak0}, the  linear condition $\rho(Q)(\xi_p^{2k+3}\odot \xi_p^{2k+3})=0 $ gives a hyperplane in $A_{k,0}$. Let us denote by $$A_{k,0,0}:=\{Q\in A_{k,0} : \rho(Q)(\xi_p^{2k+3}\odot\xi_p^{2k+3})=0\}.$$
\end{remark}

Following \cite{cfp} we give the following definition
\begin{definition}
A nonzero direction $\zeta \in H^1(T_C)^+$ is said to be asymptotic if $\rho_{HE}(Q)(\zeta \odot \zeta)=\rho(Q)(\zeta \odot \zeta) = 0$, $\forall Q \in I_2$. 
\end{definition}

Notice that saying that a nonzero element $\zeta \in H^1(T_C)^+$ is an asymptotic direction means that the point $[\zeta] \in {\mathbb P}H^1(T_C)^+$ is in the base locus of the linear space of quadrics $\rho_{HE}(I_2) \subset Sym^2 (H^1(T_C)^+)^{\vee}$. 
\begin{remark}
\label{schasym}
Recall that by Theorem \ref{thm2} putting $k=0$, a Schiffer variation at a Weiestrass point $\xi_p$ is asymptotic. 
%In fact by \eqref{mu2} we have $$\rho_{HE}(Q_{ij}) (\xi_p \odot \xi_p) = - 2 \pi i \mu_2(Q_{ij})(p),$$ 
%where $\{Q_{ij}\}$ is the basis of $I_2$ as in \eqref{basis}. We have $$\mu_2(Q_{ij}) (p) = \mu_{1,L}(s \wedge t)(p) \mu_{1, M}(\omega_i \wedge \omega_j)(p)=0,$$ since $p$ is a ramification point of the $g^1_2$. 
Moreover, moving a branch point in ${\mathbb P}^1$, one can see that there exist algebraic curves in the hyperelliptic locus having these Schiffer variations as tangent directions (see \cite{gt}). 
  
\end{remark}

\begin{theorem}
\label{asymp}
Let $C$ be hyperelliptic of genus $g \geq 3$ and $p \in C$ be a Weierstrass point. Then the asymptotic directions in the space $V_{\lfloor \frac{g-2}{2} \rfloor}$, which is equal to $\langle\xi_p^1,\xi_p^3,...,\xi_p^{g-1}\rangle$ if $g$ is even, and  equal to $\langle\xi_p^1,\xi_p^3,...,\xi_p^{g-2}\rangle$ if  $g$ is odd, are exactly the Schiffer variations $\xi_p = \xi_p^1$. 

%Let $C$ be hyperelliptic of genus $g \geq 3$ and $p \in C$ be a Weierstrass point. 

%\begin{enumerate}
%    \item  Assume $g$ is even. Then the asymptotic directions in the space $V_{\frac{g-2}{2}}= \langle\xi_p^1,\xi_p^3,...,\xi_p^{g-1}\rangle$ are exactly the Schiffer variations $\xi_p = \xi_p^1$. 
  %  \item Assume $g$ is odd. Then the  asymptotic directions in the space $V_{\frac{g-3}{2}}=\langle\xi_p^1,\xi_p^3,...,\xi_p^{g-2}\rangle$ are exactly the Schiffer variations $\xi_p = \xi_p^1$. 
%\end{enumerate}
    
\end{theorem}

\begin{proof}
By Remark \ref{schasym} we know that the Schiffer variations at the Weierstrass points are asymptotic. 

    To show that these are the only ones, we will prove the result when $g$ is even, the case $g$ odd being very similar. \\ Let $v\in \langle \xi_p^1,\xi_p^3,...,\xi_p^{g-1}\rangle$, then $v=\sum\limits_{i=0}^{\frac{g-2}{2}}\lambda_{2i+1}\xi_p^{2i+1}$ and let $l$ be the maximum odd integer such $\lambda_l\neq 0$. We may assume $l \geq 3$, since $ \rho(Q)(\xi_p \odot \xi_p ) = - 2 \pi i \mu_2(Q)(p) =0$ for all $Q \in I_2$.  Then we have $v\in \langle \xi_p^1,...,\xi_p^l\rangle $. Since $l$ is odd, we have $l=2k+1$ for some integer $1 \leq k\leq \frac{g-2}{2}$. From \Cref{thm3} we know that there exists $Q\in Ker(\mu_{2k-2})$ such that $\rho(Q)(\xi_p^{2k+1}\odot\xi_p^{2k-1})\neq 0$ and $\rho(Q)(\xi_p^i\odot\xi_p^j)= 0$ for every $i,j\leq 2k-1$, $i,j$ odd integers. Then we have 
    $$
    \rho(Q)(v\odot v)=2\lambda_{2k-1}\lambda_{2k+1}\rho(Q)(\xi_p^{2k-1}\odot \xi_p^{2k+1})+$$
    $$+\lambda_{2k+1}^2 \rho(Q)(\xi_p^{2k+1}\odot \xi_p^{2k+1}).
    $$
     First, notice that for every $k \leq \frac{g-2}{2} $ we have $A_{k-1,0} \subsetneq Ker_{\mu_{2k-2}}$ and $ A_{k-1,0,0} \subsetneq A_{k-1,0} $.  In fact, dim $Ker(\mu_{2k-2})\geq 3$, since dim $Ker(\mu_{g-4})=3$ by Theorem \ref{rteoremainrese} and $2k-2 \leq g-4$.  
     
    If $\lambda_{2k-1}=0$, we simply take $Q \in A_{k-1,0} \setminus A_{k-1,0,0}$ and we get $\rho(Q)(v\odot v)\neq 0$ as desired.\\
    If $\lambda_{2k-1}\neq 0$, we choose $Q \in Ker(\mu_{2k-2})\setminus A_{k-1,0}$ and $Q'\in A_{k-1,0}\setminus A_{k-1,0,0}$. Hence, for a linear combination $Q + a Q'$, we get
    $$
    \rho(Q+aQ')(v\odot v)=2\lambda_{2k-1}\lambda_{2k+1}\rho(Q)(\xi_p^{2k-1}\odot \xi_p^{2k+1})+$$
    $$+\lambda_{2k+1}^2[ \rho(Q)(\xi_p^{2k+1}\odot \xi_p^{2k+1})+a\rho(Q')(\xi_p^{2k+1}\odot\xi_p^{2k+1})].
    $$
    By the choice of $Q'$, we have $\rho(Q')(\xi_p^{2k+1}\odot\xi_p^{2k+1})\neq 0$, so if we take $$a=-\frac{\rho(Q)(\xi_p^{2k+1}\odot\xi_p^{2k+1})}{\rho(Q')(\xi_p^{2k+1}\odot\xi_p^{2k+1})}$$
    we get $\rho(Q+aQ')(v\odot v)=2\lambda_{2k-1}\lambda_{2k+1}\rho(Q)(\xi_p^{2k-1}\odot \xi_p^{2k+1})\neq 0$ since $Q\notin A_{k-1,0}$.
    
\end{proof}

\begin{remark}
\label{chains}
In  \cite[Corollary 3.4]{cftrans} it is shown that $\rho_{HE}$ is injective, so 
$$\dim \rho_{HE}(Ker\mu_{2k}) = \dim Ker\mu_{2k} = \frac{(g-1)(g-2)}{2} - k(2g -2k -3),$$ 
for all $k$. For any hyperelliptic curve $C$ of genus $g \geq 3$ and for any  Weiestrass point $p \in C$, we have proven in Theorems \ref{thm2}, \ref{thm3} that we have the following chain of subspaces of quadrics in ${\mathbb P} H^1(T_C)^+$ and of corresponding maximal isotropic subspaces in $V_{\lfloor \frac{g-1}{2} \rfloor}$.

If $g$ is odd: $$
\rho_{HE}(Ker(\mu_{g-3}))\subsetneq \rho_{HE}(Ker(\mu_{g-5})) ...\subsetneq \rho_{HE}(Ker(\mu_2))\subsetneq \rho_{HE}(I_2);
$$
$$V_{\frac{g-3}{2}} \supsetneq V_{\frac{g-5}{2}} \supsetneq ... \supsetneq V_1 \supsetneq V_0=\langle \xi_p \rangle,$$
while if $g$ is even we have 
$$
\rho_{HE}(Ker(\mu_{g-4}))\subsetneq \rho_{HE}(Ker(\mu_{g-6}))...\subsetneq \rho_{HE}(Ker(\mu_2))\subsetneq \rho_{HE}(I_2),
$$
$$V_{\frac{g-4}{2} }\supsetneq V_{\frac{g-6}{2}} \supsetneq ... \supsetneq V_1 \supsetneq V_0=\langle \xi_p \rangle.$$
\end{remark}

Theorem  \ref{asymp} also allows us to give a bound for the dimension of a germ of a totally geodesic submanifold of ${\mathcal A}_g$, generically contained in the Hyperelliptic Torelli locus. Notice that, nevertheless this bound is weaker than the one proven in \cite[Theorem 6.2]{fp}. 
\begin{corollary}
\label{bound}
Let $Y$ be a germ of a totally geodesic submanifold of $\mathcal{A}_g$ generically contained in $j(\mathcal{H}_g)$ passing through $j(C)$, where $C$ is a hyperelliptic curve of genus $g$. Then we have 
$$\dim (Y)\leq \lfloor \frac{3g+1}{2} \rfloor.$$ 
\end{corollary}
\begin{proof}
We will prove it for $g$ even, the case $g$ odd is analogous.\\
    Set $W:=T_{j(C)}Y\subset H_1(T_C)^+$ and $V:=V_{\frac{g}{2} -1}= \langle \xi_p^1,\xi_p^3,...,\xi_p^{g-1}\rangle$. We have dim$(V)=\frac{g}{2}$. We claim that $\dim(V\cap W) \leq 1$. Indeed, take $v\in V\cap W$, then $v$ is asymptotic, since $Y$ is totally geodesic. So, by Theorem \ref{asymp} we know that $v$ is a multiple of $\xi^1_p$. Hence we have
    \begin{center}
    dim$(V)$+dim$(W) - 1 \leq \text{dim}$(V+W)$\leq $ dim$H^1(T_C)^+=2g-1$,
    \end{center}
    so dim$(Y)$= dim $(W)\leq 2g-\frac{g}{2}=\frac{3}{2}g$.
\end{proof}

%\Addresses

\bibliographystyle{amsalpha}

\begin{thebibliography}{1}


%\bibitem{bm} A. Beauville and J. Merindol, \newblock  \emph{Sections hyperplanes des surfaces K3}, Duke Math. J. 55 (1987), \textbf{4}, 873-878.

\bibitem{bm} A. Beauville and J.-Y. M\'erindol, \emph{Sections hyperplanes des surfaces $K3$}, Duke Math. J. {\bf 55} (1987), no.~4, 873--878.

\bibitem{chm} C. Ciliberto, J.~D. Harris and R. Miranda, \emph{On the surjectivity of the Wahl map}, Duke Math. J. {\bf 57} (1988), no.~3, 829--858.



\bibitem{cm} C. Ciliberto and R. Miranda, \emph{Gaussian maps for certain families of canonical curves}, In: “Complex projective geometry” (Trieste, 1989/Bergen, 1989), 106--127, London Math. Soc. Lecture Note Ser., 179, Cambridge Univ. Press, Cambridge.

\bibitem{colombofrediani} E. Colombo and P. Frediani, \emph{Some results on the second Gaussian map for curves}, Michigan Math. J. {\bf 58} (2009), no.~3, 745--758.



\bibitem{cftrans} E. Colombo and P. Frediani, \emph{Siegel metric and curvature of the moduli space of curves}, Trans. Amer. Math. Soc. {\bf 362} (2010), no.~3, 1231--1246.

 \bibitem{cfg} E. Colombo, P. Frediani and A. Ghigi, \emph{On totally geodesic submanifolds in the Jacobian locus}, Internat. J. Math. {\bf 26} (2015), no.~1, 1550005, 21 pp.

   
\bibitem{cfp} E. Colombo, P. Frediani and G.~P. Pirola, \emph{Asymptotic directions in the moduli space of curves}, J. Reine Angew. Math. {\bf 825} (2025), 267--303.

\bibitem{cpt} E. Colombo, G.~P. Pirola and A. Tortora, \emph{Hodge-Gaussian maps}, Ann. Scuola Norm. Sup. Pisa Cl. Sci. (4) {\bf 30} (2001), no.~1, 125--146.

\bibitem{fredianihigher} P. Frediani, \emph{Second fundamental form and higher Gaussian maps}, In: “Perspectives on four decades of algebraic geometry. Vol. 1. In memory of Alberto Collino”, 289--311, Progr. Math., 351, Birkh\"auser/Springer, Cham.


\bibitem{fgp} P. Frediani, A. Ghigi and M. Penegini, \emph{Shimura varieties in the Torelli locus via Galois coverings}, Int. Math. Res. Not. IMRN {\bf 2015}, no.~20, 10595--10623.

\bibitem{fp} P. Frediani and G.~P. Pirola, \emph{On the geometry of the second fundamental form of the Torelli map}, Proc. Amer. Math. Soc. {\bf 149} (2021), no.~3, 1011--1024.

\bibitem{gt} V. Gonz\'alez-Alonso and S. Torelli, \emph{Families of curves with Higgs field of arbitrarily large kernel}, Bull. Lond. Math. Soc. {\bf 53} (2021), no.~2, 493--506.

 
\bibitem{lu-zuo} X. Lu and K. Zuo, \emph{The Oort conjecture on Shimura curves in the Torelli locus of hyperelliptic curves}, J. Math. Pures Appl. (9) {\bf 108} (2017), no.~4, 532--552.

\bibitem{moonenhyp} B. Moonen, \emph{The Coleman-Oort conjecture: reduction to three key cases}, Bull. Lond. Math. Soc. {\bf 54} (2022), no.~6, 2418--2426.



\bibitem{os} F. Oort and J.~H.~M. Steenbrink, \emph{The local Torelli problem for algebraic curves}, In: “Journ\'ees de G\'eometrie Alg\'ebrique d'Angers, Juillet 1979/Algebraic Geometry, Angers”, 1979, pp. 157--204, Sijthoff \& Noordhoff, Alphen aan den Rijn---Germantown, Md.
 
\bibitem{voi} C. Voisin, \emph{Sur l'application de Wahl des courbes satisfaisant la condition de Brill-Noether-Petri}, Acta Math. {\bf 168} (1992), no.~3-4, 249--272.

\bibitem{wahl} 
J.~M. Wahl, \emph{Gaussian maps on algebraic curves}, J. Differential Geom. {\bf 32} (1990), no.~1, 77--98.

\bibitem{wahl1} J.~M. Wahl, \emph{The Jacobian algebra of a graded Gorenstein singularity}, Duke Math. J. {\bf 55} (1987), no.~4, 843--871.


\end{thebibliography}

\end{document}